\documentclass [twoside,reqno,12pt] {amsart}
\usepackage[left=1in,right=1in,top=1in,bottom=1in]{geometry}

\usepackage[hidelinks]{hyperref}
\usepackage{amsfonts}
\usepackage{amssymb}
\usepackage{color}
\usepackage{graphics}
\usepackage{comment}

\newtheorem{thm}{Theorem}[section]
\newtheorem{cor}[thm]{Corollary}
\newtheorem{lem}[thm]{Lemma}
\newtheorem{prop}[thm]{Proposition}

\theoremstyle{definition}

\numberwithin{equation}{section}

\renewcommand{\Re}{\mathrm{Re}}
\renewcommand{\Im}{\mathrm{Im}}

\newcommand{\cA}{\mathcal{A}}

\newcommand{\C}{\mathbb{C}}

\renewcommand{\div}{\operatorname{div}}

\newcommand{\R}{\mathbb{R}}

\newcommand\scl{\mathrm{scl}}

\def\tilde{\widetilde}
\def \bfo {\begin {eqnarray*} }
\def \efo {\end {eqnarray*} }
\def \ba {\begin {eqnarray*} }
\def \ea {\end {eqnarray*} }
\def \beq {\begin {eqnarray}}
\def \eeq {\end {eqnarray}}
\def \supp {\hbox{supp }}

\def \p {\partial}

\newcommand{\cO}{\mathcal{O}}
\newcommand{\cF}{\mathcal{F}}

\def\tilde{\widetilde}
\def \bfo {\begin {eqnarray*} }
\def \efo {\end {eqnarray*} }
\def \ba {\begin {eqnarray*} }
\def \ea {\end {eqnarray*} }
\def \beq {\begin {eqnarray}}
\def \eeq {\end {eqnarray}}
\def \supp {\hbox{supp }}

\def \p {\partial}

\begin{document}
\title[Partial data inverse problem for biharmonic operator]{Stable determination of the first order perturbation of the biharmonic operator from partial data}

\author[Liu]{Boya Liu}
\address{B. Liu, Department of Mathematics\\
North Dakota State University\\ 
Fargo, ND 58102, USA}
\email{boya.liu@ndsu.edu}

\author[Selim]{Salem Selim}
\address{S. Selim, Department of Mathematics\\
University of California, Irvine\\ 
CA 92697, USA }
\email{selimsa@uci.edu}

\begin{abstract}
We consider an inverse boundary value problem for the biharmonic operator with the first order perturbation in a bounded domain of dimension three or higher. Assuming that the first and the zeroth order perturbations are known in a neighborhood of the boundary, we establish log-type stability estimates for these perturbations from a partial Dirichlet-to-Neumann map. Specifically, measurements are taken only on arbitrarily small open subsets of the boundary.
\end{abstract}

\maketitle

\section{Introduction and Statement of Results}
\label{sec:intro}

Let $\Omega \subseteq \R^n$, $n \geq 3$, be a bounded domain with smooth boundary $\p \Omega$. In this paper we study an inverse problem for the biharmonic operator with first order perturbation defined by
\begin{equation}
\label{eq:def_biharmonic}
\mathcal{L}_{A,q} (x, D):= \Delta^2+A(x) \cdot D +q(x),
\end{equation}
where $D = i^{-1} \nabla$, $A\in W^{1,\infty}(\Omega, \C^n)$, and $q\in L^\infty(\Omega, \C)$. The operator $\mathcal{L}_{A,q}$, equipped with the domain $\mathcal{D}(\mathcal{L}_{A,q}) = \{u \in H^4(\Omega): u|_{\partial \Omega} = (\Delta u)|_{\partial \Omega}= 0\}$,
is an unbounded closed operator on $L^2(\Omega)$ with purely discrete spectrum, see \cite[Chapter 11]{Grubb}. Here and in what follows the space $H^s(\Omega)=\{U|_\Omega: U\in H^s(\R^n)\}$, $s\in \R$, is the standard $L^2$-based Sobolev space on the domain $\Omega$.

The polyharmonic operator $(-\Delta)^m$, $m\ge 2$, arises in various practical scenarios. It is applied to model the equilibrium configuration of an elastic plate hinged along the boundary. In physics and geometry, polyharmonic operators appear in the study of the Kirchhoff plate equation in the theory of elasticity, the continuum mechanics of buckling problems, and the study of the Paneitz-Branson operator in conformal geometry, see \cite{Ashbaugh,Campos,Meleshko}. We refer readers to two books \cite{Gazzola_Grunau_Sweers,Selvadurai_book} for additional applications where the study of higher order operators is useful.

Consider the  boundary value problem with Navier boundary conditions
\begin{equation}
\label{eq:bvp}
\begin{cases}
\mathcal{L}_{A,q}u = 0 \quad \text{in} \quad \Omega, 
\\
u=f \quad \text{on} \quad \partial \Omega, 
\\
\Delta u=g \quad \text{on} \quad \partial \Omega,
\end{cases}
\end{equation}  
where $A \in W^{1, \infty}(\Omega, \C^n)$ and $q \in L^\infty(\Omega, \C)$. If 0 is not an eigenvalue of the operator $\mathcal{L}_{A,q}: \mathcal{D}(\mathcal{L}_{A,q}) \to L^2(\Omega)$, for any pair of functions $(f,g)\in H^{\frac{7}{2}}(\p \Omega)\times H^{\frac{3}{2}} (\p \Omega)$, the boundary value problem \eqref{eq:bvp} has a unique solution $u\in H^4(\Omega)$, see \cite{Krupchyk_Lassas_Uhlmann_poly}. 

In this paper we are concerned with a partial data inverse problem of recovering the vector field $A$ and the function $q$ from measurements made on arbitrarily small subsets of $\p \Omega$, under the assumption that $A$ and $q$ are  \textit{a priori} known in a neighborhood of $\p \Omega$. This assumption is realistic in most of the applications, since coefficients are already known or can be easily determined near the boundary, see for instance \cite{Ammari_Uhlmann,Bel_Chou}. The study of inverse problems with such assumptions was initiated in \cite{Ammari_Uhlmann} to establish uniqueness for the Schr\"odinger operator. For the perturbed biharmonic operator $\mathcal{L}_{A,q}$, uniqueness for both $A$ and $q$ was obtained in \cite[Theorem 1.3]{Yang}. We refer readers to \cite{Fathallah,Ben_Joud,Krupchyk_Uhlmann_stability,Zhao_Yuan} and the references therein to see several stability results for various elliptic operators with this partial data setting. 

Let us now describe the aforementioned partial data setting in mathematical terms.  Let $\omega_0\subset \Omega$ be an arbitrary neighborhood of $\p \Omega$. We assume that the vector field $A_0\in W^{1,\infty}(\Omega, \C^n)$ and the function $q_0\in L^\infty(\Omega, \C)$ are both known in $\omega_0$. To introduce boundary measurements, let $\nu$ be the unit outer normal to $\partial \Omega$, and let $\gamma_1, \gamma_2 \subseteq \p \Omega$ be arbitrary nonempty open sets. Associated with the boundary value problem \eqref{eq:bvp}, we define a partial  Dirichlet-to-Neumann map $\Lambda_{A,q}^{\gamma_1, \gamma_2}: H^{\frac{7}{2}}(\p \Omega)\times H^{\frac{3}{2}} (\p \Omega)\to H^{\frac{5}{2}}(\p \Omega)\times H^{\frac{1}{2}}(\p \Omega)$  by the formula
\begin{equation}
\label{eq:def_DN_map}
\begin{aligned}
\Lambda_{A,q}^{\gamma_1, \gamma_2}(f,g)=(\p_\nu u|_{\gamma_2}, \p_\nu (\Delta u)|_{\gamma_2}), \quad \supp f, \: \supp g \subseteq \gamma_1.
\end{aligned}
\end{equation}
In what follows, for any real numbers $\alpha$ and $\beta$, we shall denote $H^{\alpha, \beta}(\partial \Omega)$ the product of two Sobolev spaces $H^\alpha(\partial \Omega) \times H^\beta (\partial \Omega)$, equipped with the norm
\[
\|(f, g)\|_{H^{\alpha, \beta}(\partial \Omega)} = \|f\|_{H^\alpha(\partial \Omega)} + \|g\|_{H^\beta(\partial \Omega)}.
\]
We then define
\[
\|\Lambda_{A,q}^{\gamma_1, \gamma_2}\| := \sup \big\{\|\Lambda_{A,q}^{\gamma_1, \gamma_2}(f, g)\|_{H^{\frac{5}{2},\frac{1}{2}}(\gamma_2)}: \|(f, g)\|_{H^{\frac{7}{2}, \frac{3}{2}}(\gamma_1)} = 1\big\}.
\]
Throughout the remainder of this paper, we shall write $a\lesssim b$ to  indicate that $a\le Cb$ for some constant $C>0$  depending only on $\Omega$ and \textit{a priori} assumptions on the potentials.

For any constant $M>0$, we define the  admissible sets for the potentials $A$ and $q$ as follows:
\[
\cA(M,s, A_0,\omega_0)
=
\big\{A\in W^{1,\infty}(\Omega,\C^n):  \|A\|_{H^s(\Omega)}\le M, \: s>\frac{n}{2}+1, \text{ and } A(x)=A_0(x) \text{ in } \omega_0\big\}
\]
and
\[
\mathcal{Q}(M,q_0,\omega_0)
=
\left\{
q\in L^\infty(\overline{\Omega},\C): \|q\|_{L^\infty(\Omega)}\le M \text{ and } q(x)=q_0(x) \text{ in } \omega_0
\right\}.
\]
Similar to \cite{Liu_2024_biharmonic,Ma_Liu_stability}, we assume \textit{a priori} bounds for the $H^s$-norm of the vector field $A$ in the definition of $\mathcal{A}(M, s, A_0, \omega_0)$, where $s$ is sufficiently large and depends on the dimension of the domain $\Omega$. 

Our goal of this paper is to establish stability estimates for the vector field $A$ and the function $q$ from measurements made on arbitrary open subsets of the boundary. Specifically, we prove a log-type estimate for $A$, as well as a log-log-type estimate for $q$. Our main results in this paper complement the uniqueness results given in \cite[Theorem 1.3]{Yang}. First, the log-type stability estimate for the first order perturbation is as follows.
\begin{thm}
\label{thm:estimate_A}
Let $\Omega \subseteq \R^n, n \geq 3$, be a bounded domain with smooth connected boundary $\p \Omega$. Let $\omega_0 \subset \Omega$ be a known arbitrary neighborhood of $\partial \Omega$. Let $M \geq 0$, and let $A_j \in \cA(M,s, A_0,\omega_0)$, $q_j \in \mathcal{Q}(M,q_0,\omega_0)$, $j=1, 2$.  Suppose that 0 is not an eigenvalue for both operators $\mathcal{L}_{A_1,q_1}$ and  $\mathcal{L}_{A_2,q_2}$. Then there exist constants $\mu_1, \mu_2 \in (0,1)$ such that
\begin{equation}
\label{eq:estimate_A}
\|A_1 - A_2\|_{L^\infty(\Omega)} 
\lesssim 
\|\Lambda_{A_1, q_1}^{\gamma_1, \gamma_2}-\Lambda_{A_2, q_2}^{\gamma_1, \gamma_2}\|^{\mu_1}
+
\left|\log \|\Lambda_{A_1, q_1}^{\gamma_1, \gamma_2}-\Lambda_{A_2, q_2}^{\gamma_1, \gamma_2}\|\right|^{-\mu_2}.
\end{equation}
Here the constants $\mu_1$ and $\mu_2$ are given by  $\mu_1=\frac{\eta \tilde \eta}{6(1+s)^2}$ and $\mu_2 =\frac{\eta^2 \tilde \eta^2}{3(n+2)^2(1+s)^4}$, respectively, where $\eta = \frac{1}{2}(s-\frac{n}{2})$ and $\tilde \eta = \frac{1}{2}(s-(\frac{n}{2}+1))$. 
\end{thm}

We shall also establish the following log-log-type estimate for the zeroth order perturbation.

\begin{thm}
\label{thm:estimate_q}
Under the same hypotheses as in Theorem \ref{thm:estimate_A}, there exists a constant  $\mu' \in (0,1)$ such that
\begin{equation}
\label{eq:estimate_q}
\|q_1 - q_2\|_{H^{-1}(\Omega)} 
\lesssim 
\left|\log \left|\log\|\Lambda_{A_1, q_1}^{\gamma_1, \gamma_2}-\Lambda_{A_2, q_2}^{\gamma_1, \gamma_2}\|\right|\right|^{-\mu'}.
\end{equation}
Here $\mu' = \min \left\{ \frac{2}{n+2}, \frac{\mu_2}{2}\right\}$, where $\mu_2 \in (0,1)$ is the same as in Theorem \ref{thm:estimate_A}.
\end{thm}

If the potentials $q_1$ and $q_2$ pose additional regularity properties and \textit{a priori} bounds, we have an immediate corollary of Theorem \ref{thm:estimate_q}.

\begin{cor}
\label{cor:estimate_q_Linfty}
Let $\Omega \subseteq \R^n, n \geq 3$, be a bounded domain with smooth connected boundary $\p \Omega$. Let $\omega_0 \subset \Omega$ be a known arbitrary neighborhood of $\partial \Omega$. Let $M \geq 0$, and let $A_j \in \cA(M,s, A_0,\omega_0)$. Assume that  $q_j \in L^\infty(\Omega)$ satisfies $\|q_j\|_{H^s(\Omega)}\le M$, $j=1, 2$, where $s>\frac{n}{2}$.  Suppose that 0 is not an eigenvalue for both operators $\mathcal{L}_{A_1,q_1}$ and  $\mathcal{L}_{A_2,q_2}$. Then there exists a  constant $\mu' \in (0,1)$ such that
\begin{equation}
\label{eq:estimate_q_Linfty}
\|q_1 - q_2\|_{L^\infty(\Omega)} 
\lesssim 
\left|\log \left| \log \|\Lambda_{A_1, q_1}^{\gamma_1, \gamma_2}-\Lambda_{A_2, q_2}^{\gamma_1, \gamma_2}\| \right|^{-\mu'}\right|^{\frac{s-\frac{n}{2}}{1+s}},
\end{equation}
where $\mu'$ is the same as in Theorem \ref{thm:estimate_q}.
\end{cor}

Moreover, in the absence of the first order perturbation, the stability of the zeroth order term improves from log-log-type to log-type. 

\begin{thm}
\label{thm:estimate_q_no_A}
Let $\Omega \subseteq \R^n, n \geq 3$, be a bounded domain with smooth connected boundary $\p \Omega$. Let $\omega_0 \subset \Omega$ be a known arbitrary neighborhood of $\partial \Omega$. Let $M \geq 0$, and let $q_j \in \mathcal{Q}(M,q_0,\omega_0)$, $j=1, 2$.  Suppose that 0 is not an eigenvalue for both operators $\mathcal{L}_{0,q_1}$ and  $\mathcal{L}_{0,q_2}$. Then we have the estimate
\begin{equation}
\label{eq:estimate_q_no_A}
\|q_1 - q_2\|_{H^{-1}(\Omega)} 
\lesssim
\|\Lambda_{0, q_1}^{\gamma_1, \gamma_2}-\Lambda_{0, q_2}^{\gamma_1, \gamma_2}\|^\frac{2}{3}
+
\left| \log\|\Lambda_{0, q_1}^{\gamma_1, \gamma_2}-\Lambda_{0, q_2}^{\gamma_1, \gamma_2}\| \right|^\frac{-2}{n+2}.  
\end{equation}
\end{thm}

\begin{cor}
\label{cor:estimate_q_Linfty_no_A}
Let $\Omega \subseteq \R^n, n \geq 3$, be a bounded domain with smooth connected boundary $\p \Omega$. Let $\omega_0 \subset \Omega$ be a given arbitrary neighborhood of $\partial \Omega$. For any constant $M\ge 0$, assume that  $q_j \in L^\infty(\Omega)$ satisfies $\|q_j\|_{H^s(\Omega)}\le M$, $j=1, 2$, $s>\frac{n}{2}$.  Suppose that 0 is not an eigenvalue for both operators $\mathcal{L}_{0,q_1}$ and  $\mathcal{L}_{0,q_2}$. Then the following estimate holds:
\begin{equation}
\label{eq:estimate_q_Linfty_no_A}
\|q_1 - q_2\|_{L^\infty(\Omega)} 
\lesssim 
\left( 
\|\Lambda_{0, q_1}^{\gamma_1, \gamma_2}-\Lambda_{0, q_2}^{\gamma_1, \gamma_2}\|^\frac{2}{3}
+
\left| \log\|\Lambda_{0, q_1}^{\gamma_1, \gamma_2}-\Lambda_{0, q_2}^{\gamma_1, \gamma_2}\| \right|^\frac{-2}{n+2} \right)^{\frac{s-\frac{n}{2}}{1+s}}.
\end{equation}
\end{cor}

We next review some previous literature related to the inverse problems considered in this paper. The study of inverse boundary value problems to recover coefficients appearing in the polyharmonic operator has received significant attention in recent years. Assuming that the first order term $A=0$ in \eqref{eq:def_biharmonic}, uniqueness results of the potential $q$ were established in \cite{Ikehata,Isakov_91}. For the first order perturbation of the polyharmonic operator, uniqueness results from full boundary measurements were obtained in \cite{Krupchyk_Lassas_Uhlmann_poly} in dimension three or higher, and in  \cite{Bansal_Krishnan_Pattar} in dimension two. There are extensive subsequent efforts to uniquely recover the first order perturbation of the polyharmonic operator with lower regularity, see   \cite{Assylbekov_16,Assylbekov_Iyer,Brown_Gauthier,Krupchyk_Uhlmann_poly_16} and the references therein. Furthermore, as the polyharmonic operator is of order $2m$, it is natural to also investigate inverse boundary value problems for the second or higher order perturbations of the polyharmonic operator, see for instance \cite{Bhattacharyya_Ghosh_19,Bhattacharyya_Krishnan_Sahoo_23,Ghosh_Krishnan}.

In many practical scenarios, one can only make measurements  on parts of the boundary of a medium, as performing measurements over the entire boundary can be extremely difficult. Therefore, the study of partial data inverse problems is both significant and of great interest. It was proved in \cite{Krupchyk_Lassas_Uhlmann_bi_partial} that a partial Dirichlet-to-Neumann map, where the Dirichlet data is measured on the entire boundary and the Neumann data is measured only on approximately half of the boundary, uniquely determines the first order perturbation of the biharmonic operator. We point out that the uniqueness for the first order perturbation of the polyharmonic operator, where the Dirichlet data and the Neumann data are both measured on open subsets of the boundary, remains an important and challenging open problem. An analogous result for the magnetic Schr\"odinger operator was established in \cite{Chung_142}. In addition, several partial data uniqueness results for the first order perturbation of the biharmonic operator were obtained in \cite{Yang}, in both bounded and unbounded domains. 

Let us recall that, due to a natural gauge invariance of boundary measurements, there is an obstruction to uniqueness in the recovery of the first order perturbation of the Laplacian, which was discussed in \cite{Knudsen_Salo,Krup_Lass_Uhl_magSchr_euc,Liu_2018,Sun_95} and the references therein among the vast amount of related literature. However, as first noted in \cite{Krupchyk_Lassas_Uhlmann_poly}, this phenomenon does not occur for the first order perturbation of the biharmonic operator. Indeed, for the biharmonic operator, a gauge invariance arises in the unique recovery of linear third order perturbations, see \cite {Bhattacharyya_Ghosh_second_order} for details. This remains an open problem, and we refer readers to \cite {Bhattacharyya_Krupchyk_Sahoo_Uhlmann} for a discussion of its difficulty. Meanwhile, it was discovered in the seminal work \cite{Kurylev_Lassas_Uhlmann} that nonlinearity can be helpful to solve inverse problems for nonlinear PDEs,  even in the event that the corresponding problems for linear equations are yet solved. A very recent result \cite{Bhattacharyya_Krupchyk_Sahoo_Uhlmann} shows that the Dirichlet-to-Neumann map uniquely determines a nonlinear third order perturbation of the biharmonic operator, notably without gauge invariance. 

Another important aspect of inverse boundary value problems is the issue of stability. For the biharmonic operator, most of the known results concern only the zeroth order perturbation. The authors of \cite{Choudhury_Krishnan} established a log-type stability estimate from full data. For partial data results, a log-type estimate was proved in \cite{Choudhury_Heck} when the inaccessible part of the boundary is flat, and measurements are performed on its complement. Additionally, a log-log-type estimate was established in \cite{Choudhury_Krishnan} when the Neumann data is measured on slightly more than half of the boundary. 

Turning our attention to the first order perturbation,  we refer readers to  \cite{Ma_Liu_stability} for log-type estimates of the first and the zeroth order perturbations of the biharmonic operator from full data. To the best of our knowledge, the only result addressing the stability of the first order perturbation of the biharmonic operator with partial data is established by the first author in \cite{Liu_2024_biharmonic}, where the inaccessible portion of the boundary is flat, and measurements are made on its complement.

Let us now proceed to describe the main ideas to establish the main results of this paper. The general strategy follows from the methods introduced in \cite{Alessandrini} using complex geometric optics (CGO) solutions. We shall construct a CGO solution $u_1\in H^4(\Omega)$ to the equation $\mathcal{L}_{A_1,q_1}^\ast u_1=0$ in $\Omega$, where $\mathcal{L}_{A,q}^\ast =\mathcal{L}_{\overline{A},q-i\div \overline{A}}$ is the formal $L^2$-adjoint of $\mathcal{L}_{A_1,q_1}$, as well as a CGO solution $u_2\in H^4(\Omega)$ satisfying the equation $\mathcal{L}_{A_2,q_2} u_2=0$ in $\Omega$.  The existence of CGO solutions was established in for instance \cite[Proposition 2.4]{Krupchyk_Lassas_Uhlmann_poly}, see also \cite[Proposition 2.3]{Yang}. 

The proof of the main results also comprises several other key components. First, let us note that the Neumann data was only measured on the subset $\gamma_2$, hence it is essential to obtain the corresponding $H^s$-norms only over $\gamma_2$, rather than over the entire boundary. To achieve this, a crucial step in the proof is a local quantitative unique continuation for solutions of the nonhomogeneous boundary value problem \eqref{eq:nonhomogeneous_bvp} with homogeneous boundary conditions, see Lemma \ref{lem:unique_continuation}. To the best of our knowledge, it is the first such result for   higher order elliptic operators, and its proof relies on  Carleman estimates for the semiclassical biharmonic operator $h^4\Delta^2$, which is given in Proposition \ref{prop:Carle_est_biharmonic}. 

Another crucial ingredient of the proof is the derivation of the integral estimate \eqref{eq:integral_estimate}, which is a  consequence of the integral identity \eqref{eq:int_identity} and the local quantitative unique continuation. We remark that the derivation of \eqref{eq:int_identity} requires knowledge of the partial Dirichlet-to-Neumann map. By substituting the  CGO solutions \eqref{eq:v_form} and \eqref{eq:u2_form} into the integral estimate, we get the Fourier transform of $\mathrm{d}A_1-\mathrm{d}A_2$, as well as some error terms. Here, if we view the vector field $A$ as a one-form, its exterior derivative $\mathrm{d}A$ is a two-form defined by the formula
\begin{equation}
\label{eq:def_mag_field}
\mathrm{d}A=\sum_{1\le j<k\le n}(\p_{x_j}A_k-\p_{x_k}A_j)dx_j\wedge dx_k.
\end{equation}
We control the error terms by utilizing the decaying property of the remainder terms of the CGO solutions $u_1$ and $u_2$, which is given by \eqref{eq:est_r_domain}.  Subsequently, we decompose the $H^{-1}$-norm in terms of the Fourier transform \eqref{eq:H-1norm_split}, as well as careful choices of various parameters, to derive an estimate for $\|\mathrm{d}A_1-\mathrm{d}A_2\|_{H^{-1}(\R^n)}$. Following that, we apply the Sobolev embedding theorem and the interpolation inequality \cite[Theorem 7.22]{Grubb} to obtain an estimate for $\|\mathrm{d}A_1-\mathrm{d}A_2\|_{L^\infty(\Omega)}$, see the estimate \eqref{eq:est_dA_Linfty_domain}. We would like to recall that we assumed an \textit{a priori} bound of $\|A\|_{H^s(\Omega)}$, $s>\frac{n}{2}+1$, in the definition of  $\cA(M,s, A_0,\omega_0)$, so that we may apply the Sobolev embedding theorem in this step. 

The techniques utilized in the preceding analysis, particularly the construction of CGO solutions and the derivation of integral estimates, are   parallel to those used in the studies on the stability of elliptic operators, such as the Schr\"odinger operator \cite{Alessandrini,Fathallah} and the magnetic Schr\"odinger operator \cite{Ben_Joud,Tzou_stability}. These methodologies can also be adapted to establish stability for the biharmonic operator with zeroth order perturbations \cite{Choudhury_Heck,Choudhury_Krishnan,Liu_stability}. However, the presence of a first-order perturbation in the biharmonic operator introduces new challenges. To elaborate, the steps outlined above only enable us to deduce an estimate for $\|\mathrm{d}A_1-\mathrm{d}A_2\|_{L^\infty(\Omega)}$. Nevertheless, it is necessary to derive an estimate for $A_1-A_2$ in a suitable norm, since uniqueness of $A$ has been achieved, see for instance \cite{Krupchyk_Lassas_Uhlmann_poly}. To overcome this difficulty, we shall utilize a  decomposition of vector fields \cite[Theorem 3.3.2]{Sharafutdinov}: Every vector field $A$ can be written as $A=A^{\mathrm{sol}}+\nabla \varphi$,  where $A^{\mathrm{sol}}$ is  the \textit{solenoidal} part, and $\nabla \varphi$ is \textit{potential} part. In particular, the vector field $A^{\mathrm{sol}}$ is divergence-free and the function $\varphi$ satisfies $\varphi|_{\p \Omega}=0$. Thus, by \cite[Lemma 6.2]{Tzou_stability}, we get the inequality $\|A^{\mathrm{sol}}\|_{L^\infty(\Omega)}
\le 
\|\mathrm{d}A\|_{L^\infty(\Omega)}$.
In conjunction with the estimate \eqref{eq:est_dA_Linfty_domain}, this inequality gives us an estimate for $\|A_1^{\mathrm{sol}}-A_2^{\mathrm{sol}}\|_{L^\infty(\Omega)}$.

To complete the verification of the estimate \eqref{eq:estimate_A}, we still need an estimate for $\|\nabla \varphi\|_{L^\infty(\Omega)}$, which is given in Lemma \ref{lem:est_test_function_gradient}. To prove this result, we shall utilize the CGO solutions \eqref{eq:v_form} and \eqref{eq:u2_form} again to first deduce an estimate for $\| \varphi\|_{H^{-1}(\R^n)}$, followed by applying the Sobolev embedding theorem and the interpolation inequality to obtain an estimate for $\|\nabla \varphi\|_{L^\infty(\Omega)}$. From here, Theorem \ref{thm:estimate_A} follows easily from the estimates for $\|A^{\mathrm{sol}}\|_{L^\infty(\Omega)}$ and $\|\nabla \varphi\|_{L^\infty(\Omega)}$. 

This paper is organized as follows. In Section \ref{sec:prelim_results} we state and prove some preliminary results, which will be applied in the proof of the main results. They include the existence of CGO solutions and a local quantitative unique continuation of solutions to the boundary value problem \eqref{eq:nonhomogeneous_bvp}. We then proceed to prove Theorem \ref{thm:estimate_A} in Section \ref{sec:proof_magnetic}, followed by the verification of the log-log-type estimates of the zeroth order perturbation in Theorem \ref{thm:estimate_q}   in Section \ref{sec:proof_q}. Finally,  in Section \ref{sec:proof_q_no_A} we provide a proof for Theorem \ref{thm:estimate_q_no_A}, as well as a brief remark on the proof of Corollaries \ref{cor:estimate_q_Linfty} and \ref{cor:estimate_q_Linfty_no_A}.

\section{Preliminary Results}
\label{sec:prelim_results}

In this section we collect some preliminary results necessary to prove the main results of this paper. We first state a result concerning the existence of CGO solutions $u\in H^4(\Omega)$ to the equation $\mathcal{L}_{A,q}u=0$ in $\Omega$, followed by stating and proving a local quantitative unique continuation for solutions of the boundary value problem \eqref{eq:nonhomogeneous_bvp}.

\subsection{Existence of CGO solutions}
\label{subsec:CGO_solution}

A crucial element in the proof is the construction of CGO solutions of the form
\[
u(x,\zeta;h)=e^{\frac{x\cdot \zeta}{h}}(a(x,\zeta)+r(x,\zeta;h)).
\] 
Here $\zeta \in \C^n$ is a complex vector such that $\zeta\cdot \zeta=0$, $0<h\ll 1$ is a semiclassical parameter, $a$ is a smooth amplitude, and $r$ is a correction term that vanishes in the limit $h\to 0$. The existence of CGO solutions to the equation $\mathcal{L}_{A,q}u=0$ in $\Omega$ was established in \cite[Section 2]{Yang}, and we recall the statement for convenience of readers. In the proposition below, we equip the Sobolev space $H^1(\Omega)$ with a semiclassical norm
\[
\|v\|^2_{H^1_\scl(\Omega)}=\|v\|^2_{L^2(\Omega)}+\|hD v\|^2_{L^2(\Omega)}.
\]

\begin{prop}
\label{prop:CGO_solutions}

Let $A \in W^{1, \infty}(\Omega, \C^n), q \in L^\infty(\Omega, \C)$, and let $\zeta \in \C^n$ be a complex vector such that $\zeta \cdot \zeta = 0, \: \zeta = \zeta^{(0)}+\zeta^{(1)}$, where $\zeta^{(0)}$ is independent of $h>0$, $|\Re \zeta^{(0)}| = |\Im \zeta^{(0)}|=1$, and $\zeta^{(1)} = \mathcal{O}(h)$ as $h \to 0$. Then for all $h>0$ small enough, there exists a solution $u \in H^4(\Omega)$ to the equation $\mathcal{L}_{A,q}u=0$ in $\Omega$ of the form 
\begin{align*}
u(x, \zeta; h) = e^{\frac{ix \cdot \zeta}{h}}(a(x, \zeta^{(0)})+r(x, \zeta; h)),
\end{align*}
where the amplitude $a \in C^\infty(\overline{\Omega})$ satisfies the transport equation
\[
(\zeta^{(0)} \cdot \nabla)^2a = 0 \quad \text{in} \quad \Omega,
\]
and the remainder $r$ satisfies the estimate $\|r\|_{H^1_{\scl}(\Omega)} =\cO(h)$ as $h\to 0$.
\end{prop}

\subsection{Local quantitative unique continuation}
\label{subsec:Carleman_est}

The main purpose of this subsection is to state and prove a quantitative local unique continuation for the solutions of the boundary value problem \begin{equation}
\label{eq:nonhomogeneous_bvp}
\begin{cases}
\mathcal{L}_{A, q} w=F \quad \text{in} \quad \Omega,
\\
w=0 \quad \text{on} \quad \p \Omega,
\\
\Delta w=0 \quad \text{on} \quad \p \Omega.
\end{cases}
\end{equation}
It is a consequence of the following Carleman estimate for the semiclassical biharmonic operator $h^4\Delta^2$.

\begin{prop}
\label{prop:Carle_est_biharmonic}
Let $\Gamma \subseteq \p \Omega$ be a nonempty open set, and let the function $\psi \in C^\infty(\overline{\Omega}, \R)$ be such that $\psi \ge 0$, $|\nabla \psi|>0$ in $\overline{\Omega}$, and $\p_\nu \psi|_{\p \Omega \setminus \Gamma} \le 0$.  Let $\varphi=e^{\beta_0\psi}$, where $\beta_0 \gg 1$. Then there exists a  constant $h_0>0$ such that for all $0<h\le h_0$ and all functions $u\in H^4(\Omega)$ satisfying $u|_{\p \Omega}=\Delta u|_{\p \Omega}=0$, we have the estimate
\begin{equation}
\label{eq:Car_est_biharmonic}
h\int_{\Omega} e^{\frac{2\varphi}{h}} \left(|u|^2+|h\nabla u|^2\right)dx
\lesssim
\int_{\Omega} e^{\frac{2\varphi}{h}} |h^4\Delta^2 u|^2 dx + h \int_{\Gamma} e^{\frac{2\varphi}{h}} \left(|h \p_\nu u|^2+|h \p_\nu (h^2\Delta u)|^2\right) dS.
\end{equation}
\end{prop}

\begin{proof}
Our starting point is the following $L^2$-weighted inequality for the semiclassical Laplacian $-h^2\Delta$, which was originally proved in \cite[Lemma 3.2]{Zhao_Yuan}, see also \cite{Ben_Joud,Krupchyk_Uhlmann_stability}. For any function $u\in H^2(\Omega)$ such that $u=0$ on $\p \Omega$, we have
\begin{equation}
\label{eq:Car_est_Laplacian}
h\int_{\Omega} e^{\frac{2\varphi}{h}} \left(|u|^2+|h\nabla u|^2\right)dx
\lesssim
\int_{\Omega} e^{\frac{2\varphi}{h}} |h^2\Delta u|^2 dx + h \int_{\Gamma} e^{\frac{2\varphi}{h}} |h \p_\nu u|^2 dS.
\end{equation}

By applying \eqref{eq:Car_est_Laplacian} to the function $h^2\Delta u$ such that $\Delta u|_{\p \Omega}=0$, we get that
\begin{equation}
\label{eq:Car_est_Laplacian_apply}
h\int_{\Omega} e^{\frac{2\varphi}{h}} \left(|h^2\Delta u|^2+|h\nabla (h^2\Delta u)|^2\right)dx
\lesssim
\int_{\Omega} e^{\frac{2\varphi}{h}} |h^4\Delta^2 u|^2 dx + h \int_{\Gamma} e^{\frac{2\varphi}{h}} |h \p_\nu (h^2\Delta u)|^2 dS.
\end{equation}
From here, we obtain the estimate \eqref{eq:Car_est_biharmonic} by combining \eqref{eq:Car_est_Laplacian} and \eqref{eq:Car_est_Laplacian_apply}, where we have omitted the $|h\nabla(h^2\Delta u)|^2$ term on the left-hand side of \eqref{eq:Car_est_Laplacian_apply}. This completes the proof of Proposition \ref{prop:Carle_est_biharmonic}.
\end{proof}

We are now ready to state and prove the main result of this section. It extends \cite[Proposition 2.1]{Fathallah} for the Schr\"odinger operator, as well as \cite[Lemma 2.4]{Ben_Joud} in the case of the magnetic Schr\"odinger operator, to the setting of higher order elliptic operators. 

\begin{lem}
\label{lem:unique_continuation}
Let $\omega_j \subset \Omega$ be neighborhoods of $\p \Omega$ such that $\overline{\omega}_{j}\subset \omega_{j-1}$, $j=1,2,3$, and $\p \Omega \subset \p \omega_j$, $j=0,1,2,3$.  Let  $\Gamma_0\subset \p \Omega$ be an arbitrary open subset. Let $A \in W^{1,\infty}(\overline{\Omega}, \C^n)$ and $q \in L^\infty(\Omega, \C)$. Let $F\in L^2(\Omega)$ be a function, and let $w\in H^4(\Omega)$ be a solution of the boundary value problem \eqref{eq:nonhomogeneous_bvp}. 
Then there exist positive constants $ \alpha_1, \alpha_2$, and $h_0$ such that the estimate 
\begin{equation}
\label{eq:unique_continuation}
\|w\|_{H^1(\omega_2 \setminus \omega_3)}
\lesssim
e^{-\frac{\alpha_1}{h}}\left(\|w\|_{H^3(\Omega)}
+\|F\|_{L^2(\omega_0)}\right)
+
e^{\frac{\alpha_2}{h}} \left(\|\p_\nu w\|_{H^{\frac{5}{2}}(\Gamma_0)}
+ 
\|\p_\nu(\Delta w)\|_{H^{\frac{1}{2}}(\Gamma_0)}\right)
\end{equation}
is valid for any $0<h\le h_0$. Here  $\alpha_1$ and $\alpha_2$ depend on $\Omega, M, n, h_0$, and $\omega_j$, but are independent from $A,q, F, w$, and $h$.
\end{lem}

In order to prove Lemma \ref{lem:unique_continuation}, we need a weight function with special properties. The existence of such a function is established in for instance \cite[Lemma 1.1]{Fursikov_Imanuvilov}.

\begin{thm}
\label{thm:special_functions}
Let $\emptyset \ne \Gamma_0 \subset \p\Omega$ be an arbitrary open subset. Then there exists a function $\psi\in C^\infty(\overline{\Omega}, \R)$ such that $\psi(x)>0$ for all $x\in \Omega$, $|\nabla \psi(x)|>0$ for all $x\in \overline{\Omega}$, $\psi|_{\p\Omega \setminus \Gamma_0}=0$, and $\p_\nu \psi|_{\p\Omega \setminus \Gamma_0}\le 0$. 
\end{thm}

\begin{proof}[Proof of Lemma \ref{lem:unique_continuation}]
We shall follow the arguments in the proof of \cite[Lemma 2.4]{Ben_Joud} closely. Let $0<h\ll 1$ be a semiclassical parameter, and let us rewrite the boundary value problem \eqref{eq:nonhomogeneous_bvp} semiclassically as
\begin{equation}
\label{eq:nonhomogeneous_bvp_semiclassical}
\begin{cases}
h^4\mathcal{L}_{A, q} w = h^4F \quad \text{in} \quad \Omega,
\\
w=0 \quad \text{on} \quad \p \Omega,
\\
h^2 \Delta w=0 \quad \text{on} \quad \p \Omega.
\end{cases}
\end{equation}

By Theorem \ref{thm:special_functions}, there exists a function $\psi \in C^\infty(\overline{\omega_0}, \R)$ satisfying the properties
\[
\psi(x)>0 \text{ for all } x\in \omega_0, \quad |\nabla \psi(x)|>0 \text{ for all } x\in \overline{\omega_0}, \quad \psi|_{\p\omega_0\setminus \Gamma_0}=0, \quad \p_\nu \psi|_{\p\omega_0\setminus \Gamma_0} \le 0.
\]
Let $\beta_0>0$ be a sufficiently large constant such that the function $\varphi=e^{\beta_0\psi}$ satisfies the hypothesis of Proposition \ref{prop:Carle_est_biharmonic}.  

Due to the fact that $\psi(x)>0$ for all $x\in \omega_0$, there exists a constant $\kappa>0$ such that
\begin{equation}
\label{eq:lower_bound_psi}
\psi(x)\ge 2\kappa, \quad x\in \omega_2\setminus \omega_3. 
\end{equation}
On the other hand, since $\psi =0$ on the set $\p\omega_0\setminus \Gamma_0$, there exists a small neighborhood $\omega'$ of $\p\omega_0\setminus \Gamma_0$ such that $\omega'\cap \overline{\omega_1}=\emptyset$ and 
\begin{equation}
\label{eq:upper_bound_psi}
\psi(x)\le \kappa, \quad x\in \omega'.
\end{equation}
Let $\omega''\subset \omega'$ be an arbitrary fixed neighborhood of $\p \omega_0\setminus \p \Omega$, and let $\theta \in C^\infty(\overline{\omega_0})$ be a cut-off function such that $0\le \theta\le 1$ and  
\[
\theta(x) = 
\begin{cases}
1, \quad x\in \omega_0 \setminus \omega'',
\\
0, \quad x\in \omega''.
\end{cases}
\]
Then it follows from direct computations  that the function $ \tilde w=\theta w$, where $w$ is a solution to the boundary value problem \eqref{eq:nonhomogeneous_bvp_semiclassical}, satisfies the equation
\begin{equation}
\label{eq:nonhomo_cutoff}
\begin{cases}
h^4\mathcal{L}_{A, q}  \tilde w = h^4\theta F +Q(x,D)w \quad \text{in} \quad \omega_0,
\\
\tilde w=0 \quad \text{on} \quad \p \Omega,
\\
h^2 \Delta  \tilde w=0 \quad \text{on} \quad \p \Omega,
\end{cases}
\end{equation}
where $Q(x,D)$ is a third order differential operator supported in $\omega \setminus \omega''$ given by the formula
\[
Q(x,D)w 
= 
[h^4\Delta^2,\theta]w+h^3A\cdot[hD,\theta]w. 
\]
Here and henceforth $[A,B]=AB-BA$ denotes the commutator of two operators $A$ and $B$.

By applying the Carleman estimate \eqref{eq:Car_est_biharmonic} to the semiclassical biharmonic operator $h^4\Delta^2$ on the domain $\omega_0$ and the function $ \tilde w$ satisfying \eqref{eq:nonhomo_cutoff}, we see that there exists a positive constant $h_0$ such that the estimate
\begin{equation}
\label{eq:Car_est_biharmonic_term}
h\int_{\omega_0} e^{\frac{2\varphi}{h}} \left(| \tilde w|^2+|h\nabla  \tilde w|^2\right)dx
\lesssim
\int_{\omega_0} e^{\frac{2\varphi}{h}} |h^4\Delta^2  \tilde w|^2 dx + h \int_{\Gamma_0} e^{\frac{2\varphi}{h}} \left(|h \p_\nu  \tilde w|^2+|h \p_\nu (h^2\Delta  \tilde w)|^2\right) dS
\end{equation}
holds for all $0<h\le h_0$.

Thanks to the inequalities
\[
\int_{\omega_0} e^{\frac{2\varphi}{h}} |h^4q \tilde w|^2 dx
\lesssim
h \int_{\omega_0} e^{\frac{2\varphi}{h}} | \tilde w|^2dx 
\quad \text{and} \quad
\int_{\omega_0} e^{\frac{2\varphi}{h}} |h^3A\cdot (hD \tilde w)|^2 dx
\lesssim 
h \int_{\omega_0} e^{\frac{2\varphi}{h}} |h\nabla  \tilde w|^2 dx,  
\]
which follow immediately from the fact that $0< h \ll 1$, we may perturb \eqref{eq:Car_est_biharmonic_term} by the lower order terms $h^3A \cdot (hD)+h^4q$ to obtain
\[
h\int_{\omega_0} e^{\frac{2\varphi}{h}} \left(| \tilde w|^2+|h\nabla  \tilde w|^2\right)dx
\lesssim
\int_{\omega_0} e^{\frac{2\varphi}{h}} |h^4\mathcal{L}_{A, q}  \tilde w |^2 dx + h \int_{\Gamma_0} e^{\frac{2\varphi}{h}} \left(|h \p_\nu  \tilde w|^2+|h \p_\nu (h^2\Delta  \tilde w)|^2\right) dS.
\]
In view of \eqref{eq:nonhomo_cutoff}, we have
\begin{equation}
\label{eq:Car_est_perturbed_cutoff}
\begin{aligned}
&h\int_{\omega_0} e^{\frac{2\varphi}{h}} \left(| \tilde w|^2+|h\nabla  \tilde w|^2\right)dx
\\
&\lesssim
\int_{\omega_0} e^{\frac{2\varphi}{h}} \left(|h^4F(x) |^2+|Q(x,D)w|^2\right) dx 
+ 
h \int_{\Gamma_0} e^{\frac{2\varphi}{h}} \left(|h \p_\nu  \tilde w|^2+|h \p_\nu (h^2 \Delta  \tilde w)|^2\right) dS.
\end{aligned}
\end{equation}

Let us next estimate the integrals in \eqref{eq:Car_est_perturbed_cutoff}. To this end, we  choose the weight function $\varphi=e^{\beta_0\psi}$ with $\beta_0$ sufficiently large. For the left-hand side, due to the fact that $\theta=1$ on $\omega_2 \setminus \omega_3$, which is a subset of   $\omega_0 \setminus \omega''$, it follows from inequality \eqref{eq:lower_bound_psi} that
\begin{equation}
\label{eq:Car_est_LHS}
h\int_{\omega_0} e^{\frac{2\varphi}{h}} \left(| \tilde w|^2+|h\nabla  \tilde w|^2\right)dx 
\gtrsim
he^{\frac{2}{h}e^{2\beta_0\kappa}}\|w\|^2_{H^1_\scl(\omega_2 \setminus \omega_3)}. 
\end{equation}

Turning our attention to the right-hand side of \eqref{eq:Car_est_perturbed_cutoff}, by the definition of the cut-off function $\theta$, we see that $\supp (Q(x,D)) \subset \omega'\setminus \omega''$. Furthermore, as $Q(x,D)$ is a third order differential operator, we apply  inequality \eqref{eq:upper_bound_psi}   to deduce
\begin{equation}
\label{eq:est_Q_term}
\int_{\omega_0} e^{\frac{2\varphi}{h}}|Q(x,D)w|^2dx
\lesssim
h^4e^{\frac{2}{h}e^{\beta_0\kappa}} \|w\|^2_{H^3(\omega'\setminus \omega'')}.
\end{equation}
Similarly, we apply \eqref{eq:upper_bound_psi} again to obtain
\begin{equation}
\label{eq:est_F_term}
\int_{\omega_0} e^{\frac{2\varphi}{h}} |h^4F(x)|^2dx
\lesssim
h^4e^{\frac{2}{h}e^{\beta_0\kappa}} \|F\|^2_{L^2(\omega_0)}.
\end{equation}
Finally, to estimate the boundary terms, it follows from the trace theorem that
\begin{equation}
\label{eq:est_boundary_term}
h \int_{\Gamma_0} e^{\frac{2\varphi}{h}} \left(|h \p_\nu  \tilde w|^2 +|h \p_\nu (h^2 \Delta  \tilde w)|^2\right) dS
\lesssim
h^2 e^{\frac{2}{h}e^{\beta_0\|\psi\|_{L^\infty}}} \left(\|\p_\nu w\|^2_{H^{\frac{5}{2}}(\Gamma_0)}+ \|\p_\nu(\Delta w)\|^2_{H^{\frac{1}{2}}(\Gamma_0)}\right).
\end{equation}
Hence, we combine \eqref{eq:Car_est_LHS}--\eqref{eq:est_boundary_term} to obtain the  estimate
\begin{equation}
\label{eq:est_norms}
\begin{aligned}
e^{\frac{2}{h}e^{2\beta_0\kappa}}\|w\|^2_{H^1_\scl(\omega_2 \setminus \omega_3)}
&\lesssim
e^{\frac{2}{h}e^{\beta_0\kappa}}h^3\left(\|w\|^2_{H^3(\omega'\setminus \omega'')}+\|F\|^2_{L^2(\omega_0)}\right)
\\
&\quad +he^{\frac{2}{h}e^{\beta_0\|\psi\|_{L^\infty}}} \left(\|\p_\nu w\|^2_{H^{\frac{5}{2}}(\Gamma_0)}+ \|\p_\nu(\Delta w)\|^2_{H^{\frac{1}{2}}(\Gamma_0)}\right).
\end{aligned}
\end{equation}

We now set the  positive constants $\alpha_1$ and $\alpha_2$ by $\alpha_1= e^{2\beta_0\kappa}-e^{\beta_0\kappa} $ and $\alpha_2=e^{\beta_0\|\psi\|_{L^\infty}}-e^{2\beta_0\kappa} $, respectively. Then \eqref{eq:est_norms} reads
\[
\|w\|_{H^1_\scl(\omega_2 \setminus \omega_3)}
\lesssim e^{-\frac{\alpha_1}{h}}h^{\frac{3}{2}}\left(\|w\|_{H^3(\Omega)}+\|F\|_{L^2(\omega_0)}\right)
+
he^{\frac{\alpha_2}{h}}\left(\|\p_\nu w\|_{H^{\frac{5}{2}}(\Gamma_0)}
+ 
\|\p_\nu(\Delta w)\|_{H^{\frac{1}{2}}(\Gamma_0)}\right).
\]
From here, we obtain the claimed estimate \eqref{eq:unique_continuation} by passing to the non-semiclassical $H^1$-norm of $w$ 
via the inequality
\[
\|w\|_{H^1(\omega_2 \setminus \omega_3)}
\le
h^{-1}\|w\|_{H^1_\scl(\omega_2 \setminus \omega_3)}.	
\]
This completes the proof of Lemma \ref{lem:unique_continuation}. 
\end{proof}

\section{Proof of Theorem \ref{thm:estimate_A}}
\label{sec:proof_magnetic}

This section is devoted to verifying the estimate \eqref{eq:estimate_A}, thus proving Theorem \ref{thm:estimate_A}.   For  simplicity, let us denote $A=A_2-A_1$ and $q=q_2-q_1$ for the rest of this section. We extend the vector field $A$ and the function  $q$ by zero on $\R^n \setminus \Omega$, and  denote the extensions by the same letters. Since $A_1=A_2$ and $q_1=q_2$ near $\p \Omega$, we have $A\in W^{1,\infty}(\R^n, \C^n)$ and $q\in L^\infty(\R^n, \C)$. 

The proof of Theorem \ref{thm:estimate_A} is long and consists of three main steps, which we present in following subsections. 

\subsection{Integral estimate}
\label{subsec:integral_identity}
Our starting point to prove a  stability estimate  for the inverse problem under consideration is the derivation of the following integral estimate.  
\begin{prop}
\label{prop:integral_inequality}
Let $A_j\in W^{1,\infty}(\Omega, \C^n)$, $q_j\in L^\infty(\Omega, \C)$, $j=1,2$, and let  $\Lambda_{A_j, q_j}^{\gamma_1, \gamma_2}$ be the partial Dirichlet-to-Neumann map defined by \eqref{eq:def_DN_map} associated with the operator $\mathcal{L}_{A_j, q_j}$. Then the estimate 
\begin{equation}
\label{eq:integral_estimate}
\begin{aligned}
\left|\int_\Omega \left(A\cdot Du_2+qu_2\right)\overline{u_1}dx\right|
\lesssim
&\|u_1\|_{L^2(\Omega)} \left[e^{-\frac{\alpha_1}{3h}}\|u_2\|_{H^1(\Omega)}\right.
\\
&\left.+e^{\frac{\alpha_2}{3h}} \|u_2\|_{H^1(\Omega)}^{\frac{2}{3}} \|\Lambda_{A_1, q_1}^{\gamma_1, \gamma_2}-\Lambda_{A_2, q_2}^{\gamma_1, \gamma_2}\|^{\frac{1}{3}}  \left(\|u_2\|_{H^4(\Omega)}^{\frac{1}{3}}+\|\Delta u_2\|^{\frac{1}{3}}_{H^2(\Omega)}\right)\right]
\end{aligned}
\end{equation}
holds for any functions $u_1, u_2 \in H^4(\Omega)$ that satisfy the equations $\mathcal{L}_{A_1, q_1}^\ast u_1=0$ and $\mathcal{L}_{A_2, q_2}u_2=0$ in $\Omega$, where $\mathcal{L}_{A_1, q_1}^\ast$ is the formal $L^2$-adjoint of the operator $\mathcal{L}_{A_1, q_1}$. Here the constants $\alpha_1$ and $\alpha_2$ are the same as in Lemma \ref{lem:unique_continuation}.
\end{prop}

\begin{proof}
We shall begin the proof by recalling the Green's formula for biharmonic operators, which holds for any functions $u,v\in H^4(\Omega)$, see \cite[Section 3]{Yang}:
\begin{equation}
\label{eq:Green_id}
\begin{aligned}
\int_\Omega\mathcal{L}_{A,q}u\overline{v}dx-\int_{\Omega}u\overline{\mathcal{L}_{A,q}^\ast v}dx = 
&-i\int_{\p \Omega}\nu \cdot uA\overline{v}dS
+
\int_{\p \Omega}\p_\nu (\Delta u)\overline{v}dS
-
\int_{\p \Omega}\Delta u\overline{\p_\nu v}dS
\\
&+\int_{\p \Omega}\p_\nu u\overline{\Delta v}dS
-
\int_{\p \Omega}u\overline{\p_\nu(\Delta v)}dS.
\end{aligned}
\end{equation}
Here $\nu$ is the outward unit normal to the boundary $\p \Omega$.

Let $(f,g)\in H^{\frac{7}{2}, \frac{3}{2}}(\p \Omega)$ be a pair of functions, and let the function $u_2\in H^4(\Omega)$ be a solution to the boundary value problem  
\begin{equation}
\label{eq:bvp_u2}
\begin{cases}
\mathcal{L}_{A_2,q_2}u_2 = 0 \quad \text{in} \quad \Omega, 
\\
u_2=f \quad \text{on} \quad \partial \Omega, 
\\
\Delta u_2=g \quad \text{on} \quad \partial \Omega.
\end{cases}
\end{equation}
On the other hand, let the function $v\in H^4(\Omega)$ solve the boundary value problem
\[
\begin{cases}
	\mathcal{L}_{A_1,q_1}v = 0 \quad \text{in} \quad \Omega, 
	\\
	v=f \quad \text{on} \quad \partial \Omega, 
	\\
	\Delta v=g \quad \text{on} \quad \partial \Omega.
\end{cases}
\]    
Then it follows immediately that the function $u :=v-u_2$ satisfies  
\begin{equation}
\label{eq:difference}
\begin{cases}
\mathcal{L}_{A_1, q_1}u=A \cdot Du_2+ qu_2 \quad  \text{in} \quad  \Omega,
\\
u=0 \quad \text{on} \quad  \p \Omega,
\\
\Delta u=0 \quad \text{on} \quad  \p \Omega.
\end{cases}
\end{equation}

Let $\omega_0 \subset \Omega$ be a neighborhood of $\p \Omega$ with smooth boundary $\p \omega_0$ such that $(A_1, q_1)=(A_2, q_2)$ in $\omega_0$. Let $\omega_j \subset \Omega$, $j=1,2,3$, be neighborhoods of $\p \Omega$ such that $\p \Omega \subset \p \omega_j$ and $\overline{\omega_j} \subset \omega_{j-1}$. Furthermore, let $\chi \in C_0^\infty(\Omega)$ be a cut-off function such that $0\le \chi \le 1$ and 
\[
\chi(x) = 
\begin{cases}
1, \quad x\in \Omega \setminus \omega_2,
\\
0, \quad x\in \omega_3.
\end{cases}
\]
Setting the function $\tilde u = \chi u$, we conclude through direct computations that $\tilde u$ is a solution to the equation
\[
\mathcal{L}_{A_1,q_1} \tilde u = \chi(A \cdot Du_2+ qu_2)+P(x,D)u \quad \text{in} \quad \Omega,
\]
where $P(x,D)$ is a third order differential operator supported in $\omega_2\setminus \omega_3$ defined by the formula
\[
P(x,D)= [\Delta^2,\chi] + A_1\cdot [D,\chi]. 
\]
Let us  observe that $A=0$ and $q=0$ in $\omega_0$, hence we have $\chi A=A$ and $\chi q=q$ in $\Omega$. This yields the equation
\[
\mathcal{L}_{A_1,q_1} \tilde u = A \cdot Du_2+ qu_2+P(x,D)u \quad \text{in} \quad \Omega.
\]

Let $u_1 \in H^4(\Omega)$ be a solution of the equation $\mathcal{L}_{A_1, q_1}^\ast u_1=0$ in $\Omega$. Since $\tilde u=\Delta \tilde u=0$ in $\omega_3$, we have $\p_\nu \tilde u|_{\p \Omega}=\p_\nu (\Delta \tilde u)|_{\p \Omega}=0$. Thus, by applying the Green's identity \eqref{eq:Green_id} to the functions $\tilde u, u_1\in H^4(\Omega)$, we get
\begin{equation}
\label{eq:int_identity}
\int_\Omega (A\cdot Du_2+qu_2 )\overline{u_1} dx
=
\int_{\Omega} -\overline{u_1}  P(x,D)u dx.
\end{equation}

We next estimate the integral on the right-hand side of \eqref{eq:int_identity}. As $P(x,D)$ is a third order operator with $\supp P(x,D)\subset \omega_2 \setminus \omega_3$, we apply the Cauchy-Schwarz inequality to obtain the estimate
\begin{equation}
\label{eq:est_aux_term}
\left|\int_\Omega \overline{u_1}  P(x,D)u dx\right| 
\lesssim
\|u\|_{H^3(\omega_2 \setminus \omega_3)} \|u_1\|_{L^2(\Omega)}.
\end{equation}
By the interpolation inequality and Lemma \ref{lem:unique_continuation}, we deduce that
\begin{equation}
\label{eq:interpolation}
\begin{aligned}
\|u\|_{H^3(\omega_2 \setminus \omega_3)} 
\le 
&\|u\|^{\frac{1}{3}}_{H^1(\omega_2 \setminus \omega_3)} \|u\|^{\frac{2}{3}}_{H^4(\omega_2 \setminus \omega_3)}
\\
\le 
&\|u\|^{\frac{1}{3}}_{H^1(\omega_2 \setminus \omega_3)} \|u\|^{\frac{2}{3}}_{H^4(\Omega)}
\\
\lesssim
& \|u\|^{\frac{2}{3}}_{H^4(\Omega)} \left[e^{-\frac{\alpha_1}{3h}}\left(\|u\|^{\frac{1}{3}}_{H^3(\Omega)}+\|A\cdot Du_2+qu_2\|^{\frac{1}{3}}_{L^2(\omega_0)}\right)\right.
\\
&\left. +e^{\frac{\alpha_2}{3h}}\left(\|\p_\nu u\|_{H^{\frac{5}{2}}(\gamma_2)}^{\frac{1}{3}}+ \|\p_\nu(\Delta u)\|_{H^{\frac{1}{2}}(\gamma_2)}^{\frac{1}{3}}\right)\right].
\end{aligned}
\end{equation}

Let us now analyze the terms on the right-hand side of \eqref{eq:interpolation}. Since  0 is not an eigenvalue of  $\mathcal{L}_{A_1, q_1}$, an application of the elliptic regularity \cite[Corollary 2.21]{Gazzola_Grunau_Sweers} indicates that any solution $u\in H^4(\Omega)$ to the boundary value problem \eqref{eq:difference} satisfies the estimate
\begin{equation}
\label{eq:est_H4_term}
\begin{aligned}
\|u\|_{H^4(\Omega)}
&\lesssim
\|A\cdot Du_2+qu_2\|_{L^2(\Omega)}
\\
&\lesssim
\left(\|A\|_{L^\infty(\Omega)}\|Du_2\|_{L^2(\Omega)} + \|q\|_{L^\infty(\Omega)}\|u_2\|_{L^2(\Omega)}\right)
\\
&\lesssim
\|u_2\|_{H^1(\Omega)}.
\end{aligned}
\end{equation}
Furthermore, the inequalities
\begin{equation}
\label{eq:est_H3_term}
\|u\|_{H^3(\Omega)}
\lesssim 
\|u_2\|_{H^1(\Omega)}
\end{equation}
and
\begin{equation}
\label{eq:est_RHS_term}
\|A\cdot Du_2+qu_2\|_{L^2(\omega_0)} 
\le 
\|A\cdot Du_2+qu_2\|_{L^2(\Omega)} 
\lesssim 
\|u_2\|_{H^1(\Omega)}
\end{equation}
follow immediately from  \eqref{eq:est_H4_term}.

Turning attention to the boundary terms in the estimate \eqref{eq:interpolation}, we follow similar computations as in \cite{Choudhury_Heck,Liu_2024_biharmonic} to get that
\begin{equation}
\label{eq:est_boundary_norm}
\begin{aligned}
\|\p_\nu u\|_{H^{\frac{5}{2}}(\gamma_2)}+ \|\p_\nu(\Delta u)\|_{H^{\frac{1}{2}}(\gamma_2)}
&\lesssim 
\|(\Lambda_{A_1, q_1}^{\gamma_1, \gamma_2}-\Lambda_{A_2, q_2}^{\gamma_1, \gamma_2})(f,g)\|_{H^{\frac{5}{2}, \frac{1}{2}}(\Gamma)}
\\
&\lesssim 
\|\Lambda_{A_1, q_1}^{\gamma_1, \gamma_2}-\Lambda_{A_2, q_2}^{\gamma_1, \gamma_2}\| \|(f,g)\|_{H^{\frac{7}{2}, \frac{3}{2}}(\Gamma)}
\\
&\lesssim 
\|\Lambda_{A_1, q_1}^{\gamma_1, \gamma_2}-\Lambda_{A_2, q_2}^{\gamma_1, \gamma_2}\| (\|u_2\|_{H^4(\Omega)}+\|\Delta u_2\|_{H^2(\Omega)}).
\end{aligned}
\end{equation}
Therefore, the estimate \eqref{eq:integral_estimate} follows immediately from the estimates \eqref{eq:est_aux_term}--\eqref{eq:est_boundary_norm}. This completes the proof of Proposition \ref{prop:integral_inequality}.
\end{proof}

\subsection{Estimating the exterior derivative  of the vector field}
\label{subsec:est_dA}

In this subsection we derive an estimate for $\|\mathrm{d}A\|_{L^\infty(\Omega)}$, where $\mathrm{d}A$ is the exterior derivative of the vector field $A$ given by \eqref{eq:def_mag_field}. To achieve this goal, we  first construct CGO solutions $u_1, u_2$ satisfying the equations $\mathcal{L}_{A_1, q_1}^\ast u_1=0$ and $\mathcal{L}_{A_1, q_1} u_2=0$ in $\Omega$, respectively, followed by substituting them into the integral estimate \eqref{eq:integral_estimate}. 

To construct CGO solutions, let $\mu^{(1)}, \mu^{(2)}, \xi$ be vectors in $\R^n$ such that
\[
\mu^{(1)}\cdot\mu^{(2)}=\mu^{(1)}\cdot \xi=\mu^{(2)}\cdot \xi=0, \quad |\mu^{(1)}|=|\mu^{(2)}|=1,
\]
and let us set
\begin{equation}
\label{eq:form_zeta}
\begin{aligned}
\zeta_1=\frac{\tau\xi}{2}+\sqrt{1-\tau^2\frac{|\xi|^2}{4}}\mu^{(1)}+i \mu^{(2)}, \quad
\zeta_2=-\frac{\tau\xi}{2}+\sqrt{1-\tau^2\frac{|\xi|^2}{4}} \mu^{(1)}-i \mu^{(2)},
\end{aligned}
\end{equation}
where $0<\tau\ll 1$ is a semiclassical parameter such that $1-\tau^2 \frac{|\xi|^2}{4} \ge 0$. Then direct computations show that $\zeta_j\cdot\zeta_j=0$, $j=1,2$, and $\frac{\zeta_2-\overline{\zeta_1}}{\tau}=-\xi$. Moreover, we observe that $\zeta_1 = \mu^{(1)} + i\mu^{(2)} +\mathcal{O}(\tau)$ and $\zeta_2 = \mu^{(1)} - i\mu^{(2)} + \mathcal{O}(\tau)$ as $\tau \to 0$.  From the definition of $\zeta^{(0)}$ given in Proposition \ref{prop:CGO_solutions}, we have that $\zeta_1^{(0)}= \mu^{(1)} + i\mu^{(2)}$ and $\zeta_2^{(0)}= \mu^{(1)} - i\mu^{(2)}$.

An application of Proposition \ref{prop:CGO_solutions} yields that there exists a solution $u_1\in H^4(\Omega)$ to the equation $\mathcal{L}_{A_1, q_1}^\ast u_1=0$ in $\Omega$ of the form
\begin{equation}
\label{eq:v_form}
u_1(x, \zeta_1;\tau )
=
e^{\frac{ix\cdot \zeta_1}{\tau}} \left(a_1\left(x,  \mu^{(1)} + i\mu^{(2)}\right)+r_1(x, \zeta_1;\tau)\right).
\end{equation}
Also, the equation $\mathcal{L}_{A_2, q_2} u_2=0$ in $\Omega$ admits a solution $u_2\in H^4(\Omega)$ given by
\begin{equation}
\label{eq:u2_form}
u_2(x, \zeta_2; \tau)
=
e^{\frac{ix\cdot \zeta_2}{\tau}} \left(a_2 \left(x, \mu^{(1)}-i\mu^{(2)}\right)+r_2(x, \zeta_2; \tau) \right).
\end{equation}
Here, for $j=1,2$, the amplitude $a_j\in C^\infty(\overline{\Omega})$ solves the transport equation
\begin{equation}
\label{eq:trans_eq_u1_u2}
((\mu^{(1)} + i\mu^{(2)}) \cdot \nabla)^2a_j(x, \mu^{(1)} + i\mu^{(2)}) = 0,
\end{equation} 
and the remainder term $r_j\in H^4(\Omega)$ satisfies the estimate 
\begin{equation}
\label{eq:est_r_domain}
\|r_j\|_{H^1_{\scl}(\Omega)}=\cO(\tau), \quad \tau\to 0.
\end{equation}

We next substitute the CGO solutions \eqref{eq:v_form} and \eqref{eq:u2_form} into the integral estimate \eqref{eq:integral_estimate}, multiply both sides of the resulted inequality by $\tau$, and analyze the limits as $\tau \to 0$. For the first term on the left-hand side of \eqref{eq:integral_estimate}, it follows from direct computations that
\[
Du_2 = \frac{\zeta_2}{\tau}e^\frac{ix\cdot \zeta_2}{\tau}(a_2+r_2) + e^\frac{ix\cdot \zeta_2}{\tau}(Da_2+Dr_2).
\]
Thus, we obtain that 
\begin{equation}
\label{eq:int_id_substitute_A}
\begin{aligned}
\tau\int_\Omega (A\cdot  Du_2) \overline{u_1}dx
&= \int_\Omega A\cdot \zeta_2 e^{-ix\cdot \xi}\overline{a_1}a_2dx 
+\int_\Omega A\cdot w_1 dx
\\
&=: I_1+I_2,
\end{aligned}
\end{equation}
where 
\begin{align*}
w_1 &= \zeta_2 e^{-ix\cdot \xi}(\overline{a_1}r_2+a_2\overline{r_1}+\overline{r_1}r_2)
+\tau e^{-ix\cdot \xi}(\overline{a_1}Da_2+\overline{a_1}Dr_2+\overline{r_1}Da_2+\overline{r_1}Dr_2).
\end{align*}

Let us now investigate the limit of each integral in \eqref{eq:int_id_substitute_A} as $\tau \to 0$. For $I_1$, we recall the definition of $\zeta_2$ in \eqref{eq:form_zeta} to deduce that
\begin{equation}
\label{eq:I1_sum}
I_1
\to 
(\mu^{(1)}-i\mu^{(2)}+\cO(\tau))\cdot \int_{\Omega}  A e^{-ix\cdot \xi}\overline{a_1}a_2dx, \quad \tau \to 0.
\end{equation}
By replacing the vector $\mu^{(2)}$ above by $-\mu^{(2)}$, we have
\begin{equation}
\label{eq:I1_difference}
I_1\to (\mu^{(1)}+i\mu^{(2)}+\cO(\tau))\cdot \int_{\Omega}  A e^{-ix\cdot \xi}\overline{a_1}a_2dx, \quad \tau \to 0.
\end{equation}
Hence, in view of \eqref{eq:I1_sum} and \eqref{eq:I1_difference}, we see that
\[
I_1\to (\mu+\cO(\tau)) \cdot \int_{\Omega} A e^{-ix\cdot \xi}\overline{a_1}a_2dx, \quad \tau \to 0,
\]
whenever $\mu,\xi \in \R^n$ are  such that $\mu\cdot \xi=0$.  By choosing $a_1=a_2=1$, which clearly satisfy the   transport equation   \eqref{eq:trans_eq_u1_u2}, we get that
\[
I_1 \to (\mu+\cO(\tau)) \cdot \int_{\Omega} A e^{-ix\cdot \xi}dx, \quad \tau \to 0.
\]

To deduce an estimate for the integral $I_1$, it is straightforward to see that
\[
\left|\cO(\tau)\cdot \int_\Omega Ae^{ix\cdot \xi}dx\right|
\lesssim
\tau.
\]
On the other hand, following \cite{Kru_Uhl_14}, by choosing $\mu_{jk}(\xi)=\xi_je_k-\xi_ke_j$ for $j\ne k$, where $\{e_1,\dots, e_n\}$ is the standard basis of $\R^n$, we see that $\mu_{jk}(\xi)\cdot \xi=0$, and
\begin{align*}
\mu_{jk}(\xi) \cdot \int_\Omega Ae^{ix \cdot \xi}dx 
&= 
\xi_j \mathcal{F}(A_k)-\xi_k\mathcal{F}(A_j)
= 
\mathcal{F}(\p_{x_j}A_k-\p_{x_k}A_j).
\end{align*}
Summing it over all $1\le j<k\le n$, we conclude that
\begin{equation}
\label{eq:est_I1}
|I_1|
\lesssim
|\mathcal{F}(\mathrm{d}A)(\xi)|+\tau,
\end{equation}
where $\mathrm{d}A$ is a two-form defined  by formula \eqref{eq:def_mag_field}.

Turning attention to the integral $I_2$, we first observe that the estimate \eqref{eq:est_r_domain} implies the estimate
\begin{equation}
\label{eq:est_remainder_derivative}
\|Dr_j\|_{L^2(\Omega)}=\cO(1), \: j=1,2, \quad \tau \to 0.
\end{equation}
We then utilize our choices of amplitudes $a_1=a_2=1$, the estimates \eqref{eq:est_r_domain} and \eqref{eq:est_remainder_derivative}, in conjunction with the Cauchy-Schwarz inequality, to get that
\begin{equation}
\label{eq:est_I2}
|I_2| \lesssim \tau, \quad \tau \to 0.
\end{equation}

We next estimate the second term on the left-hand side of the estimate \eqref{eq:integral_estimate}. Since $a_j\in C^\infty(\overline{\Omega})$, $j=1,2$, we apply the estimate  \eqref{eq:est_r_domain} and the Cauchy-Schwarz inequality to conclude that
\begin{equation}
\label{eq:est_qterm}
\left|\tau \int_\Omega q \overline{u_1} u_2 dx\right| \lesssim \tau, \quad \tau \to 0.
\end{equation}

To analyze the terms on the right-hand side of the estimate \eqref{eq:integral_estimate}, we   need to estimate the $H^s$-norms of the functions $u_1$, $u_2$, as well as their derivatives. To this end, by similar computations as in \cite[Section 4]{Choudhury_Krishnan}, we have the following estimates
\begin{equation}
\label{eq:est_solutions}
\|u_2\|_{H^1(\Omega)}
\lesssim
\frac{1}{\tau}e^{\frac{2R}{\tau}}, 
\quad 
\|u_1\|_{L^2(\Omega)} 
\lesssim
e^{\frac{2R}{\tau}}, 
\quad 
\|u_2\|_{H^4(\Omega)}
\lesssim
\frac{1}{\tau^4}e^{\frac{2R}{\tau}}, 
\quad 
\|\Delta u_2\|_{H^2(\Omega)}
\lesssim 
\frac{1}{\tau}e^{\frac{2R}{\tau}}.
\end{equation}
Therefore, by utilizing estimates  \eqref{eq:est_I1}, together with  \eqref{eq:est_I2}--\eqref{eq:est_solutions}, we deduce from the integral estimate \eqref{eq:integral_estimate} that
\begin{equation}
\label{eq:est_Fourier_dA}
\begin{aligned}
|\mathcal{F}(\mathrm{d}A)(\xi)|
&\lesssim
\tau + e^{\frac{2R}{\tau}} \left[e^{-\frac{\alpha_1}{3h}} e^{\frac{2R}{\tau}}
+e^{\frac{\alpha_2}{3h}} \|\Lambda_{A_1, q_1}^{\gamma_1, \gamma_2}-\Lambda_{A_2, q_2}^{\gamma_1, \gamma_2}\|^{\frac{1}{3}}  \frac{1}{\tau^{2/3}}e^{\frac{4R}{3\tau}} \left( \frac{1}{\tau^{4/3}}e^{\frac{2R}{3\tau}}+ \frac{1}{\tau^{1/3}}e^{\frac{2R}{3\tau}}\right)\right]
\\
&\lesssim 
\tau
+
e^{\frac{4R}{\tau}-\frac{\alpha_1}{3h}}
+e^{\frac{6R}{\tau}+\frac{\alpha_2}{3h}} \|\Lambda_{A_1, q_1}^{\gamma_1, \gamma_2}-\Lambda_{A_2, q_2}^{\gamma_1, \gamma_2}\|^{\frac{1}{3}},
\quad  \tau \to 0,
\end{aligned}
\end{equation}
where we have utilized the fact that $0<\tau \ll 1$ and the inequality $\frac{1}{\tau} \le e^{\frac{2R}{\tau}}$ in the last step.

Let us now connect the two semiclassical parameters $h$ and $\tau$ by setting $\tau=\lambda h$, where $\lambda>0$ is a constant. By choosing $\lambda$ sufficiently large, we see that there exist positive constants $\alpha_3$ and $\alpha_4$ such that
\begin{equation}
\label{eq:exponential_tau_h}
e^{\frac{4R}{\tau}-\frac{\alpha_1}{3h}}=e^{\frac{1}{h}\left(\frac{4R}{\lambda}-\frac{\alpha_1}{3}\right)} \le e^{-\frac{\alpha_3}{h}},
\quad
e^{\frac{6R}{\tau}+\frac{\alpha_2}{3h}}=e^{\frac{1}{h}\left(\frac{6R}{\lambda}+\frac{\alpha_2}{3}\right)} \le e^{\frac{\alpha_4}{h}}.
\end{equation}
Hence, the estimate \eqref{eq:est_Fourier_dA} becomes
\begin{equation}
\label{eq:est_Fourier_dA_new}
|\mathcal{F}(\mathrm{d}A)(\xi)|
\lesssim
h
+
e^{-\frac{\alpha_3}{h}}
+
e^{\frac{\alpha_4}{h}}\|\Lambda_{A_1, q_1}^{\gamma_1, \gamma_2}-\Lambda_{A_2, q_2}^{\gamma_1, \gamma_2}\|^{\frac{1}{3}},
\quad h\to 0.
\end{equation}

We are now ready to derive an estimate for  $\|\mathrm{d}A\|_{H^{-1}(\Omega)}$. Let $\rho>0$ be a parameter to be specified later, then an application of the Parseval's formula gives us
\begin{equation}
\label{eq:H-1norm_split}
\|\mathrm{d}A\|_{H^{-1}(\R^n)}^2
\le 
\int_{|\xi|\le \rho} \frac{|\cF(\mathrm{d}A)(\xi)|^2}{1+|\xi|^2}d\xi 
+
\int_{|\xi|\ge \rho} \frac{|\cF(\mathrm{d}A)(\xi)|^2}{1+|\xi|^2}d\xi.
\end{equation}

Let us estimate the terms on the right-hand side of \eqref{eq:H-1norm_split}. For the first term, it follows from \eqref{eq:est_Fourier_dA_new} that
\begin{equation}
	\label{eq:est_dA_small_rho}
	\begin{aligned}
		\int_{|\xi|\le \rho} \frac{|\cF(\mathrm{d}A)(\xi)|^2}{1+|\xi|^2}d\xi 
		&\lesssim 
		\left(e^{\frac{2\alpha_4}{h}}\|\Lambda_{A_1, q_1}^{\gamma_1, \gamma_2}-\Lambda_{A_2, q_2}^{\gamma_1, \gamma_2}\|^{\frac{2}{3}}
		+
		h^2
		+
		e^{-\frac{2\alpha_3}{h}}\right) \int_{|\xi|\le \rho} \frac{1}{1+|\xi|^2}d\xi 
		\\
		&\lesssim
		\rho^n \left(e^{\frac{2\alpha_4}{h}}\|\Lambda_{A_1, q_1}^{\gamma_1, \gamma_2}-\Lambda_{A_2, q_2}^{\gamma_1, \gamma_2}\|^{\frac{2}{3}}
		+
		h^2
		+
		e^{-\frac{2\alpha_3}{h}}\right)
		\\
		&\lesssim
		\rho^n \left(e^{\frac{2\alpha_4}{h}}\|\Lambda_{A_1, q_1}^{\gamma_1, \gamma_2}-\Lambda_{A_2, q_2}^{\gamma_1, \gamma_2}\|^{\frac{2}{3}}
		+
		h^2
		+
		h\right)
		\\
		&\lesssim
		\rho^n \left(e^{\frac{2\alpha_4}{h}}\|\Lambda_{A_1, q_1}^{\gamma_1, \gamma_2}-\Lambda_{A_2, q_2}^{\gamma_1, \gamma_2}\|^{\frac{2}{3}}
		+
		h\right),
	\end{aligned}
\end{equation}
where we have applied the inequalities $e^{-\frac{2\alpha_3}{h}}\lesssim h$ and $h^2<h$ for $h$ small in the last step.

To estimate the second term, we get from the Plancherel theorem that 
\begin{equation}
\label{eq:est_dA_large_rho}
\begin{aligned}
\int_{|\xi|\ge \rho} \frac{|\cF(\mathrm{d}A)(\xi)|^2}{1+|\xi|^2}d\xi 
&\lesssim
\int_{|\xi|\ge \rho} \frac{|\cF(\mathrm{d}A)(\xi)|^2}{1+\rho^2}d\xi
\\
&\lesssim
\frac{1}{\rho^2}\|\mathrm{d}A\|_{L^2(\R^n)}
\\
&\lesssim 
\frac{1}{\rho^2}.
\end{aligned}
\end{equation}
Hence, we get from the previous two inequalities that
\[
\|\mathrm{d}A\|_{H^{-1}(\R^n)}^2
\lesssim 
\rho^n e^{\frac{2\alpha_4}{h}} \|\Lambda_{A_1, q_1}^{\gamma_1, \gamma_2}-\Lambda_{A_2, q_2}^{\gamma_1, \gamma_2}\|^{\frac{2}{3}}
+
\rho^nh
+
\frac{1}{\rho^2}.
\]

We now choose $\rho = h^{-\frac{1}{n+2}}$, which equates $\rho^n h$ and $\frac{1}{\rho^2}$. Then there exists a constant $\alpha_5>\alpha_4>0$ such that $\rho^n e^{\frac{2\alpha_4}{h}}\le e^{\frac{2\alpha_5}{h}}$. Therefore, the inequality above reads
\begin{equation}
\label{eq:est_dA}
\|\mathrm{d}A\|_{H^{-1}(\R^n)}
\lesssim 
e^{\frac{\alpha_5}{h}} \|\Lambda_{A_1, q_1}^{\gamma_1, \gamma_2}-\Lambda_{A_2, q_2}^{\gamma_1, \gamma_2}\|^{\frac{1}{3}}
+
h^{\frac{1}{n+2}}.
\end{equation}

Finally, to obtain an estimate for $\|\mathrm{d}A\|_{L^\infty(\Omega)}$, let us recall that we have assumed an \textit{a priori} bound for $\|A\|_{H^s(\Omega)}$, with $s>\frac{n}{2}+1$, in the definition of the admissible set $\mathcal{A}(M, s, A_0, \omega_0)$. Thus, there exists a constant $\eta>0$ such that $s=\frac{n}{2}+2\eta$. Applying the Sobolev embedding theorem and the interpolation theorem, we deduce from the estimate \eqref{eq:est_dA} that 
\begin{equation}
\label{eq:est_dA_Linfty}
\begin{aligned}
\|\mathrm{d}A\|_{L^\infty(\R^n)}
&\lesssim 
\|\mathrm{d}A\|_{H^{\frac{n}{2}+ \eta}(\R^n)}
\\
&\lesssim  \|\mathrm{d}A\|_{H^{-1}(\R^n)}^{\frac{\eta}{1+s}} \|\mathrm{d}A\|_{H^s(\R^n)}^{\frac{1+s-\eta}{1+s}}
\\
&\lesssim  \|\mathrm{d}A\|_{H^{-1}(\R^n)}^{\frac{\eta}{1+s}}
\\
&\lesssim  
e^{\frac{\alpha_5}{h}} \|\Lambda_{A_1, q_1}^{\gamma_1, \gamma_2}-\Lambda_{A_2, q_2}^{\gamma_1, \gamma_2}\|^{\frac{\eta}{3(1+s)}}
+
h^{\frac{ \eta}{(n+2)(1+s)}}.
\end{aligned}
\end{equation}
Therefore, we conclude that
\begin{equation}
\label{eq:est_dA_Linfty_domain}
\|\mathrm{d}A\|_{L^\infty(\Omega)}
\lesssim  
e^{\frac{\alpha_5}{h}} \|\Lambda_{A_1, q_1}^{\gamma_1, \gamma_2}-\Lambda_{A_2, q_2}^{\gamma_1, \gamma_2}\|^{\frac{\eta}{3(1+s)}}
+
h^{\frac{ \eta}{(n+2)(1+s)}}.
\end{equation}

\subsection{Estimating the vector field}
\label{subsec:est_A}

We complete the proof of Theorem \ref{thm:estimate_A} in this subsection. The key step is to apply a  decomposition of a vector field \cite[Theorem 3.3.2]{Sharafutdinov}, which we state below for the convenience of readers.
\begin{lem}
\label{lem:vf_decomposition}
Let $\Omega\subseteq\R^n$, $n\ge 3$, be a bounded domain with smooth boundary $\p \Omega$, and let $A\in W^{1,\infty}(\Omega)$ be a complex-valued vector field. Then there exist uniquely determined vector field $A^{\mathrm{sol}}\in W^{1,\infty}(\Omega)$ and function $\varphi\in W^{2,\infty}(\Omega)$ such that
\[
A=A^{\mathrm{sol}}+\nabla \varphi, \quad \div A^{\mathrm{sol}}=0, \quad \varphi|_{\p \Omega}=0.
\]
Furthermore, the vector field $A^{\mathrm{sol}}$ and the function $\varphi$ satisfy the inequalities 
\[
\|\varphi\|_{W^{2,\infty}(\Omega)}\le C\|A\|_{W^{1,\infty}(\Omega)} 
\quad \text{ and } \quad 
\|A^{\mathrm{sol}}\|_{W^{1,\infty}(\Omega)}\le C\|A\|_{W^{1,\infty}(\Omega)},
\]
where the constant $C>0$ is independent of $A$.
\end{lem}

We remark that the function $\varphi$ in Lemma \ref{lem:vf_decomposition} is a solution to the boundary value problem
\[
\begin{cases}
\Delta \varphi=\div A \quad \text{in} \quad \Omega,
\\
\varphi=0 \quad \text{on} \quad \p \Omega.
\end{cases}
\]
By elliptic regularity, for any vector field $A\in H^s(\Omega)$, we have $\varphi \in H^{s+1}(\Omega)$, which satisfies the estimate
\begin{equation}
\label{eq:est_phi_and_A}
\|\varphi\|_{H^s(\Omega)}
\lesssim
\|A\|_{H^s(\Omega)}.
\end{equation}

On the other hand, by applying  \cite[Lemma 6.2]{Tzou_stability}, we get that
\begin{equation}
\label{eq:est_solenoidal}
\|A^{\mathrm{sol}}\|_{L^\infty(\Omega)}
\lesssim
\|\mathrm{d}A\|_{L^\infty(\Omega)}.
\end{equation}
It then follows from the triangle inequality and \eqref{eq:est_solenoidal} that
\begin{align*}
\|A\|_{L^\infty(\Omega)}
&\le 
\|A^{\mathrm{sol}}\|_{L^\infty(\Omega)}+\|\nabla \varphi\|_{L^\infty(\Omega)} 
\\
&\lesssim 
\|\mathrm{d}A\|_{L^\infty(\Omega)}+\|\nabla \varphi\|_{L^\infty(\Omega)}.
\end{align*}
We have already obtained an estimate for $\|\mathrm{d}A\|_{L^\infty(\Omega)}$ in the estimate \eqref{eq:est_dA_Linfty_domain}. Hence, in order to complete the proof of Theorem \ref{thm:estimate_A}, it suffices to provide an estimate for $\|\nabla \varphi\|_{L^\infty(\Omega)}$, which is given in the following result. 

\begin{lem}
\label{lem:est_test_function_gradient}
Let $\varphi\in W^{2,\infty}(\Omega)$ be the function from Lemma \ref{lem:vf_decomposition}, and let $\eta, \tilde \eta>0$ be constants such that $s=\frac{n}{2}+2\eta$ and $s-1=\frac{n}{2}+2 \tilde \eta$. Then there exist  constants $\alpha_6>0$ and $h_0>0$ such that the estimate
\begin{equation}
\label{eq:Linfty_norm_gradient_test_function}
\|\nabla \varphi\|_{L^\infty(\Omega)} 
\lesssim 
e^{\frac{\alpha_6}{h}}\|\Lambda_{A_1, q_1}^{\gamma_1, \gamma_2}-\Lambda_{A_2, q_2}^{\gamma_1, \gamma_2}\|^{\frac{ \eta \tilde \eta}{3(1+s)^2}}
+
h^{\frac{2 \eta \tilde \eta}{(n+2)^2(1+s)^2}}
\end{equation}
holds for all $h<h_0$.  
\end{lem}

\begin{proof}
We shall argue similarly as in the proof of \cite[Lemma 3.6]{Ma_Liu_stability} and \cite[Lemma 3.4]{Liu_2024_biharmonic}. We extend $\varphi$ by zero to $\R^n\setminus \Omega$ and denote the extension by the same letter. Our starting point is the analysis of the integrals appearing in \eqref{eq:int_id_substitute_A} in the limit $h\to 0$. By utilizing the same arguments leading from \eqref{eq:I1_sum}  to \eqref{eq:est_Fourier_dA_new}, we obtain
\begin{equation}
\label{eq:I1I2_est}
\left|(\mu^{(1)}-i\mu^{(2)})\cdot \int_{\Omega} A e^{-ix\cdot \xi}\overline{a_1}a_2dx\right|
\lesssim 
h
+
e^{\frac{\alpha_4}{h}}\|\Lambda_{A_1, q_1}^{\gamma_1, \gamma_2}-\Lambda_{A_2, q_2}^{\gamma_1, \gamma_2}\|^{\frac{1}{3}},
\end{equation}
where the constant $\alpha_4>0$ satisfies the second inequality in \eqref{eq:exponential_tau_h}. 

Let us now choose  $a_1=1$, and $a_2$ that satisfies the equation
\begin{equation}
\label{eq:transport_a2}
((\mu^{(1)} - i\mu^{(2)}) \cdot \nabla)a_2(x, \mu^{(1)} - i\mu^{(2)})=1 \quad \text{in} \quad \Omega.
\end{equation}
By direct computations, our choice of $a_1$ and $a_2$ satisfies the   transport equation   \eqref{eq:trans_eq_u1_u2}. Applying Lemma \ref{lem:vf_decomposition}, we rewrite the left-hand side of \eqref{eq:I1I2_est} as
\begin{align*}
&(\mu^{(1)}-i\mu^{(2)}) \cdot \int_{\Omega}  A e^{-ix\cdot \xi}a_2dx
\\
&= (\mu^{(1)}-i\mu^{(2)}) \cdot \int_{\Omega} A^{\mathrm{sol}} e^{-ix\cdot \xi}a_2dx
+
(\mu^{(1)}-i\mu^{(2)}) \cdot \int_{\Omega}  \nabla \varphi e^{-ix\cdot \xi}a_2dx.
\end{align*}
Since the amplitude  $a_2\in C^\infty(\overline{\Omega})$, we have
\begin{equation}
\label{eq:est_sol_A}
\left|(\mu^{(1)}-i\mu^{(2)})\cdot \int_{\Omega}  A^{\mathrm{sol}} e^{-ix\cdot \xi}a_2dx\right|
\lesssim  
\|A^{\mathrm{sol}}\|_{L^\infty(\Omega)}.
\end{equation}
On the other hand, as $a_2$ solves  equation \eqref{eq:transport_a2}, the function $\varphi$ satisfies $\varphi|_{\p \Omega}=0$, and $\mu^{(1)}\cdot \xi=\mu^{(2)}\cdot \xi=0$, we integrate by parts to get that
\[
(\mu^{(1)}-i\mu^{(2)})\cdot \int_{\Omega}  \nabla \varphi e^{-ix\cdot \xi}a_2dx=-\int_{\Omega} \varphi e^{-ix\cdot \xi}dx=-\mathcal{F}(\varphi)(\xi).
\]
Therefore, we obtain from estimates \eqref{eq:I1I2_est} and \eqref{eq:est_sol_A} that
\begin{equation}
\label{eq:est_phi_Fourier}
|\mathcal{F}(\varphi)(\xi)|
\lesssim 
h
+
e^{\frac{\alpha_4}{h}}\|\Lambda_{A_1, q_1}^{\gamma_1, \gamma_2}-\Lambda_{A_2, q_2}^{\gamma_1, \gamma_2}\|^{\frac{1}{3}}
+
\|A^{\mathrm{sol}}\|_{L^\infty(\Omega)}.
\end{equation}
Thanks to estimates \eqref{eq:est_dA_Linfty_domain} and \eqref{eq:est_solenoidal}, we have
\[
\|A^\mathrm{sol}\|_{L^\infty(\Omega)}
\lesssim 
e^{\frac{\alpha_5}{h}} \|\Lambda_{A_1, q_1}^{\gamma_1, \gamma_2}-\Lambda_{A_2, q_2}^{\gamma_1, \gamma_2}\|^{\frac{\eta}{3(1+s)}}
+
h^{\frac{ \eta}{(n+2)(1+s)}}.
\]
Hence, due to  the inequalities $0<h\ll 1$, $\alpha_5>\alpha_4>0$, and $\frac{\eta}{1+s}<1$, we get from \eqref{eq:est_phi_Fourier} that
\begin{equation}
\label{eq:est_Fourier_phi}
|\mathcal{F}(\varphi)(\xi)|
\lesssim 
e^{\frac{\alpha_5}{h}}\|\Lambda_{A_1, q_1}^{\gamma_1, \gamma_2}-\Lambda_{A_2, q_2}^{\gamma_1, \gamma_2}\|^{\frac{\eta}{3(1+s)}}
+
h^{\frac{\eta}{(n+2)(1+s)}}.
\end{equation}

We next proceed similarly as in Subsection \ref{subsec:est_dA}  to deduce an estimate for $\|\varphi\|_{H^{-1}(\Omega)}$. To this end, let $\rho>0$ be a parameter that we shall choose later. Applying the Parseval's formula again, we write
\[
\|\varphi\|_{H^{-1}(\R^n)}^2
\le 
\int_{|\xi|\le \rho}\frac{|\cF(\varphi)(\xi)|^2}{1+|\xi|^2}d\xi 
+
\int_{|\xi|\ge \rho}\frac{|\cF(\varphi)(\xi)|^2}{1+|\xi|^2}d\xi.
\]
By the same computations as in the estimate \eqref{eq:est_dA_large_rho}, we have
\begin{equation}
\label{eq:est_phi_large_rho}
\int_{|\xi|\ge \rho}\frac{|\cF(\varphi)(\xi)|^2}{1+|\xi|^2}d\xi
\lesssim 
\frac{1}{\rho^2}.
\end{equation}
On the other hand,  we utilize the same computations as in the estimate \eqref{eq:est_dA_small_rho} to deduce from \eqref{eq:est_Fourier_phi} that
\begin{equation}
\label{eq:est_phi_small_rho}
\int_{|\xi|\le \rho}\frac{|\cF(\varphi)(\xi)|^2}{1+|\xi|^2}d\xi 
\lesssim  
\rho^n\left(e^{\frac{2\alpha_5}{h}}\|\Lambda_{A_1, q_1}^{\gamma_1, \gamma_2}-\Lambda_{A_2, q_2}^{\gamma_1, \gamma_2}\|^{\frac{2\eta}{3(1+s)}}
+
h^{\frac{2\eta}{(n+2)(1+s)}}\right).
\end{equation}
We then conclude from \eqref{eq:est_phi_large_rho} and \eqref{eq:est_phi_small_rho} that
\[
\|\varphi\|_{H^{-1}(\R^n)}^2
\lesssim  \rho^ne^{\frac{2\alpha_5}{h}}\|\Lambda_{A_1, q_1}^{\gamma_1, \gamma_2}-\Lambda_{A_2, q_2}^{\gamma_1, \gamma_2}\|^{\frac{2\eta}{3(1+s)}}
+
\rho^nh^{\frac{2\eta}{(n+2)(1+s)}}
+
\frac{1}{\rho^2}.
\]

Choosing $\rho>0$ such that $\rho^nh^{\frac{2  \eta}{(n+2)(1+s)}}=\frac{1}{\rho^2}$, namely, $\rho=h^{-\frac{2\eta}{(n+2)^2(1+s)}}$, we get
\begin{align*}
\|\varphi\|_{H^{-1}(\R^n)}^2 
&\lesssim 
h^{-\frac{2n  \eta}{(n+2)^2(1+s)}} e^{\frac{2\alpha_5}{h}}\|\Lambda_{A_1, q_1}^{\gamma_1, \gamma_2}-\Lambda_{A_2, q_2}^{\gamma_1, \gamma_2}\|^{\frac{2\eta}{3(1+s)}}
+h^{\frac{4 \eta}{(n+2)^2(1+s)}}
\\
&\lesssim 
e^{\frac{2\alpha_6}{h}}\|\Lambda_{A_1, q_1}^{\gamma_1, \gamma_2}-\Lambda_{A_2, q_2}^{\gamma_1, \gamma_2}\|^{\frac{2\eta}{3(1+s)}}+h^{\frac{4  \eta}{(n+2)^2(1+s)}},
\end{align*}
for some constant $\alpha_6>\alpha_5>0$ such that $h^{-\frac{2n  \eta}{(n+2)^2(1+s)}} e^{\frac{2\alpha_5}{h}}\le e^{\frac{2\alpha_6}{h}}$. This gives us
\begin{equation}
\label{eq:est_phi_H-1}
\|\varphi\|_{H^{-1}(\R^n)} 
\lesssim 
e^{\frac{\alpha_6}{h}}\|\Lambda_{A_1, q_1}^{\gamma_1, \gamma_2}-\Lambda_{A_2, q_2}^{\gamma_1, \gamma_2}\|^{\frac{\eta}{3(1+s)}}
+
h^{\frac{2\eta}{(n+2)^2(1+s)}}.
\end{equation}

To complete the proof of the lemma, let us set $\tilde \eta = \frac{1}{2}(s-(\frac{n}{2}+1))>0$. Using the Sobolev embedding theorem and the interpolation theorem, we deduce from estimates \eqref{eq:est_phi_and_A} and  \eqref{eq:est_phi_H-1}  that
\begin{align*}
	\|\nabla \varphi\|_{L^\infty(\Omega)}
	&\le
	\|\nabla \varphi\|_{L^\infty(\R^n)}
	\\
	&\lesssim 
	\|\nabla \varphi\|_{H^{\frac{n}{2}+ \tilde \eta}(\R^n)}
	\\
	&\lesssim  
	\|\nabla \varphi\|_{H^{-2}(\R^n)}^{\frac{\tilde \eta}{1+s}} \|\nabla \varphi\|_{H^{s-1}(\R^n)}^{\frac{1+s-\tilde \eta}{1+s}}
	\\
	&\lesssim  
	\|\nabla \varphi\|_{H^{-2}(\R^n)}^{\frac{\tilde \eta}{1+s}} \|\varphi\|_{H^{s}(\R^n)}^{\frac{1+s-\tilde \eta}{1+s}}
	\\
	&\lesssim
	\|\nabla \varphi\|_{H^{-2}(\R^n)}^{\frac{\tilde \eta}{1+s}}
	\\
	&\lesssim 
	\|\varphi\|_{H^{-1}(\R^n)}^{\frac{\tilde \eta}{1+s}}
	\\
	&\lesssim
	e^{\frac{\alpha_6}{h}}\|\Lambda_{A_1, q_1}^{\gamma_1, \gamma_2}-\Lambda_{A_2, q_2}^{\gamma_1, \gamma_2}\|^{\frac{\eta \tilde \eta}{3(1+s)^2}}
	+
	h^{\frac{2 \eta \tilde \eta}{(n+2)^2(1+s)^2}},
\end{align*}
where we have applied the inequalities $e^{\frac{\alpha_6}{h}} \gg 1$ for $h$ small and $\frac{\tilde \eta}{1+s}<1$ in the last step. This completes the proof of Lemma \ref{lem:est_test_function_gradient}.
\end{proof}

We are now ready to complete the proof of Theorem \ref{thm:estimate_A}.
We first assume that
$\|\Lambda_{A_1, q_1}^{\gamma_1, \gamma_2}-\Lambda_{A_2, q_2}^{\gamma_1, \gamma_2}\| 
\ll
1.$
Due to the inequalities $0<h\ll 1$, $\frac{\tilde \eta}{1+s}<1$, and $\frac{2\tilde \eta}{(n+2)(1+s)}< 1$, we combine estimates \eqref{eq:est_dA_Linfty_domain}, \eqref{eq:est_solenoidal}, and \eqref{eq:Linfty_norm_gradient_test_function} to deduce that
\begin{equation}
\label{eq:est_A_Linfty}
\begin{aligned}
\|A\|_{L^\infty(\Omega)}
&\lesssim 
e^{\frac{\alpha_5}{h}} \|\Lambda_{A_1, q_1}^{\gamma_1, \gamma_2}-\Lambda_{A_2, q_2}^{\gamma_1, \gamma_2}\|^{\frac{\eta}{3(1+s)}}
+
h^{\frac{ \eta}{(n+2)(1+s)}}
+
e^{\frac{\alpha_6}{h}}\|\Lambda_{A_1, q_1}^{\gamma_1, \gamma_2}-\Lambda_{A_2, q_2}^{\gamma_1, \gamma_2}\|^{\frac{\eta \tilde \eta}{3(1+s)^2}}
+
h^{\frac{2 \eta \tilde \eta}{(n+2)^2(1+s)^2}}
\\
&\lesssim  
e^{\frac{\alpha_6}{h}} \|\Lambda_{A_1, q_1}^{\gamma_1, \gamma_2}-\Lambda_{A_2, q_2}^{\gamma_1, \gamma_2}\|^{\frac{ \eta \tilde \eta}{3(1+s)^2}}
+
h^{\frac{2 \eta \tilde \eta}{(n+2)^2(1+s)^2}}.
\end{aligned}
\end{equation}

Taking $h>0$ sufficiently small,  we argue similarly as in \cite[Section 3]{Liu_2024_biharmonic} to conclude that there exists a constant $\delta>0$ such that  $\|\Lambda_{A_1, q_1}^{\gamma_1, \gamma_2}-\Lambda_{A_2, q_2}^{\gamma_1, \gamma_2}\| <\delta$ implies $h\le h_0$. Let us then  choose
\[
h= \alpha_6\big( \big|\log \|\Lambda_{A_1, q_1}^{\gamma_1, \gamma_2}-\Lambda_{A_2, q_2}^{\gamma_1, \gamma_2}\|\big|\big)^{-\frac{ \eta \tilde \eta }{6(1+s)^2}}
\]
and substitute it into \eqref{eq:est_A_Linfty} to deduce that
\[
\|A\|_{L^\infty(\Omega)}
\lesssim 
\|\Lambda_{A_1, q_1}^{\gamma_1, \gamma_2}-\Lambda_{A_2, q_2}^{\gamma_1, \gamma_2}\|^{\frac{ \eta \tilde \eta}{6(1+s)^2}}
+
\big|\log \|\Lambda_{A_1, q_1}^{\gamma_1, \gamma_2}-\Lambda_{A_2, q_2}^{\gamma_1, \gamma_2}\|\big|^{-\frac{\eta^2 \tilde \eta^2}{3(n+2)^2(1+s)^4}}.
\]

On the other hand, if $\|\Lambda_{A_1, q_1}^{\gamma_1, \gamma_2}-\Lambda_{A_2, q_2}^{\gamma_1, \gamma_2}\| > \delta$, we have
\[
\|A\|_{L^\infty(\Omega)} 
\lesssim 
M
\lesssim 
\frac{M}{\delta^{\frac{ \eta \tilde \eta}{6(1+s)^2}}}\delta^{\frac{ \eta \tilde \eta}{6(1+s)^2}}
\lesssim  \frac{M}{\delta^{\frac{ \eta \tilde \eta}{6(1+s)^2}}}\|\Lambda_{A_1, q_1}^{\gamma_1, \gamma_2}-\Lambda_{A_2, q_2}^{\gamma_1, \gamma_2}\|^{\frac{ \eta \tilde \eta}{6(1+s)^2}}.
\]
This completes the proof of Theorem \ref{thm:estimate_A}.

\section{Proof of Theorem \ref{thm:estimate_q}}
\label{sec:proof_q}

Our  goal of  this section is to verify the stability estimate concerning the zeroth order perturbation $q$. Let us  again denote $A=A_2-A_1$ and $q=q_2-q_1$. We also extend $q$ by zero on $\mathbb{R}^n \setminus \Omega$ and denote the extension by the same letter. 

We shall begin the proof by recalling the integral identity \eqref{eq:int_identity}:
\[
\int_\Omega (A\cdot Du_2 +qu_2)\overline{u_1}dx
= 
\int_\Omega -\overline{u_1}P(x,D)u,
\]
where $u_1\in H^4(\Omega)$ is a solution to the equation  $\mathcal{L}_{A_1,q_1}^*u_1=0$ in $\Omega$, and $u_2$ and $u$ satisfy   boundary value problems \eqref{eq:bvp_u2} and \eqref{eq:difference}, respectively. Moreover, same as in Section \ref{sec:proof_magnetic}, $P(x,D)$ is a third order operator defined by the formula
\[
P(x,D)= [\Delta^2,\chi] + A_1\cdot [D,\chi]. 
\]

In view of Proposition \ref{prop:integral_inequality} and the estimate \eqref{eq:est_solutions}, we perform similar computations as in the estimate \eqref{eq:est_Fourier_dA} to obtain that
\begin{equation}
\label{eq:int_est_q}
\left|\int_\Omega \left(A\cdot Du_2+qu_2\right)\overline{u_1}dx\right|
\lesssim  
e^{\frac{4R}{\tau}-\frac{\alpha_1}{3h}}
+
e^{\frac{6R}{\tau}+\frac{\alpha_2}{3h}} \|\Lambda_{A_1, q_1}^{\gamma_1, \gamma_2}-\Lambda_{A_2, q_2}^{\gamma_1, \gamma_2}\|^{\frac{1}{3}},
\quad  \tau \to 0. 
\end{equation}
Furthermore, the estimate \eqref{eq:est_solutions}, in conjunction with the Cauchy-Schwarz inequality and the trace theorem, gives us
\begin{equation}
\label{eq:est_A_term}
\begin{aligned}
\left|\int_\Omega (A\cdot Du_2)\overline{u_1}dx\right| 
&\le \|A\|_{L^\infty(\Omega)}\|Du_2\|_{L^2(\Omega)}\|u_1\|_{L^2(\Omega)}
\\
&\le
\|A\|_{L^\infty(\Omega)} \|u_2\|_{H^1(\Omega)}\|u_1\|_{L^2(\Omega)}
\\
&\lesssim 
\|A\|_{L^\infty(\Omega)} \frac{1}{\tau}e^{\frac{2R}{\tau}}e^\frac{2R}{\tau}
\\
&\lesssim 
\|A\|_{L^\infty(\Omega)}e^\frac{5R}{\tau}, \quad \tau \to 0,
\end{aligned}
\end{equation}
where we have applied the inequality $\frac{1}{\tau}\le e^\frac{R}{\tau}$ in the last step. Hence, it follows from  \eqref{eq:int_est_q} and \eqref{eq:est_A_term} that
\begin{equation}
\label{eq:integral_est_elec_potential_1}
\bigg|\int_\Omega qu_2\overline{u_1}dx\bigg| \lesssim  
e^{\frac{4R}{\tau}-\frac{\alpha_1}{3h}}
+
e^{\frac{6R}{ \tau}+\frac{\alpha_2}{3h}} \|\Lambda_{A_1, q_1}^{\gamma_1, \gamma_2}-\Lambda_{A_2, q_2}^{\gamma_1, \gamma_2}\|^{\frac{1}{3}} 
+
e^\frac{5R}{\tau} \|A\|_{L^\infty(\Omega)},
\quad  \tau \to 0.
\end{equation}
Here the constants $\alpha_1, \alpha_2>0$ are the same as in Lemma \ref{lem:unique_continuation}.

We next substitute the CGO solutions $u_1$ and $u_2$ given by \eqref{eq:v_form} and \eqref{eq:u2_form}, with amplitudes $a_1=a_2=1$, into the left-hand side of \eqref{eq:integral_est_elec_potential_1}. Then we obtain from straightforward computations that
\[
\int_\Omega qu_2\overline{u_1}dx =  \mathcal{F}(q)(\xi)+I,
\]
where 
\[
I = \int_\Omega qe^{-ix\cdot \xi}(\overline{r_1}+r_2+r_2\overline{r_1})dx.
\]
From here, an application of the estimate \eqref{eq:est_r_domain} and the Cauchy-Schwarz inequality gives us the estimate
\begin{equation}
\label{eq:est_rem_q}
|I| \lesssim  \tau, \quad \tau \to 0.
\end{equation}
Thus, \eqref{eq:integral_est_elec_potential_1} and \eqref{eq:est_rem_q} yield that
\[
|\mathcal{F}(q)(\xi)| 
\lesssim 
\tau
+
e^{\frac{4R}{\tau}-\frac{\alpha_1}{3h}}
+
e^{\frac{6R}{\tau}+\frac{\alpha_2}{3h}} \|\Lambda_{A_1, q_1}^{\gamma_1, \gamma_2}-\Lambda_{A_2, q_2}^{\gamma_1, \gamma_2}\|^{\frac{1}{3}} 
+
e^\frac{5R}{\tau}\|A\|_{L^\infty(\Omega)},
\quad  \tau \to 0.
\]
Furthermore, by Theorem \ref{thm:estimate_A}, we get that
\begin{equation}
\label{eq:integral_est_elec_potential_3}
\begin{aligned}
|\mathcal{F}(q)(\xi)| 
&\lesssim 
\tau+
e^{\frac{4R}{\tau}-\frac{\alpha_1}{3h}}
+e^{\frac{6R}{\tau}+\frac{\alpha_2}{3h}} \|\Lambda_{A_1, q_1}^{\gamma_1, \gamma_2}-\Lambda_{A_2, q_2}^{\gamma_1, \gamma_2}\|^{\frac{1}{3}} +e^{\frac{5R}{\tau}} \|\Lambda_{A_1, q_1}^{\gamma_1, \gamma_2}-\Lambda_{A_2, q_2}^{\gamma_1, \gamma_2}\|^{\mu_1} 
\\
& \: \: \: \: +
e^\frac{5R}{\tau} \left|\log\|\Lambda_{A_1, q_1}^{\gamma_1, \gamma_2}-\Lambda_{A_2, q_2}\|\right|^{-\mu_2}. 
\end{aligned}
\end{equation}
Let us now set $\tau = \lambda h$, where $\lambda >0$ is a constant large enough so that both inequalities in \eqref{eq:exponential_tau_h} hold. We also assume that $ \|\Lambda_{A_1, q_1}^{\gamma_1, \gamma_2}-\Lambda_{A_2, q_2}^{\gamma_1, \gamma_2}\| \ll 1$ for the moment. Since  $\mu_1 = \frac{\eta^2}{6(1+s)^2}<\frac{1}{3}$, we deduce from \eqref{eq:integral_est_elec_potential_3} that
\begin{equation}
\label{eq:integral_est_elec_potential_4}
\begin{aligned}
|\mathcal{F}(q)(\xi)| 
&\lesssim 
h+
e^{-\frac{\alpha_3}{h}}
+e^\frac{\alpha_4}{h} \|\Lambda_{A_1, q_1}^{\gamma_1, \gamma_2}-\Lambda_{A_2, q_2}^{\gamma_1, \gamma_2}\|^{\frac{1}{3}} +e^{\frac{\alpha_7}{h}} \|\Lambda_{A_1, q_1}^{\gamma_1, \gamma_2}-\Lambda_{A_2, q_2}^{\gamma_1, \gamma_2}\|^{\mu_1}
\\
&  \: \: \: \: +
e^\frac{\alpha_5}{h}\left|\log\|\Lambda_{A_1, q_1}^{\gamma_1, \gamma_2}-\Lambda_{A_2, q_2}\|\right|^{-\mu_2}
\\
&\lesssim  
h+e^{\frac{\alpha_8}{h}} \|\Lambda_{A_1, q_1}^{\gamma_1, \gamma_2}-\Lambda_{A_2, q_2}^{\gamma_1, \gamma_2}\|^{\mu_1}
+
e^\frac{\alpha_5}{h} \left|\log\|\Lambda_{A_1, q_1}^{\gamma_1, \gamma_2}-\Lambda_{A_2, q_2}\|\right|^{-\mu_2}.
\end{aligned}
\end{equation}
Here the constants $\alpha_3$ and $\alpha_4$ are the same as in \eqref{eq:exponential_tau_h}, $\alpha_5 = \frac{5R}{\lambda}$, and $\alpha_8 = \max\{\alpha_4,\alpha_7\}$. We have also used the inequality $e^{-\frac{\alpha_3}{h}} \le h$ for $h$ sufficiently small in the last step.

We next move to establish an  estimate for $\|q\|_{H^{-1}(\mathbb{R}^n)}$. To this end, let $\rho >0$ be a parameter to be chosen later. By the Parseval's formula, we write
\[
\|q\|_{H^{-1}(\mathbb{R}^n)}^2 \le \int_{|\xi|\le \rho} \dfrac{|\mathcal{F}(q)(\xi)|^2}{1+|\xi|^2}d\xi + \int_{|\xi|\geq \rho} \dfrac{|\mathcal{F}(q)(\xi)|^2}{1+|\xi|^2}d\xi.
\]
Arguing similarly as in the estimate \eqref{eq:est_dA_large_rho}, we get that
\begin{equation}
\label{eq:est_q_in_H-1_part_1}
\int_{|\xi|\geq \rho} \dfrac{|\mathcal{F}(q)(\xi)|^2}{1+|\xi|^2}d\xi 
\lesssim  
\frac{1}{\rho^2}.
\end{equation}
On the other hand, it follows from the estimate \eqref{eq:integral_est_elec_potential_4} that
\begin{equation}
\label{eq:est_q_in_H-1_part_2}
\int_{|\xi|\le \rho} \frac{|\mathcal{F}(q)(\xi)|^2}{1+|\xi|^2}d\xi 
\lesssim 
\rho^n \left(h^2+e^{\frac{2\alpha_8}{h}} \|\Lambda_{A_1, q_1}^{\gamma_1, \gamma_2}-\Lambda_{A_2, q_2}^{\gamma_1, \gamma_2}\|^{2\mu_1}+
e^\frac{2\alpha_7}{h} \left|\log\|\Lambda_{A_1, q_1}^{\gamma_1, \gamma_2}-\Lambda_{A_2, q_2}\|\right|^{-2\mu_2}\right).
\end{equation}
Hence, by combining \eqref{eq:est_q_in_H-1_part_1} and \eqref{eq:est_q_in_H-1_part_2}, we have
\[
\|q\|_{H^{-1}(\mathbb{R}^n)}^2 
\lesssim 
\rho^nh^2+\rho^n e^{\frac{2\alpha_8}{h}} \|\Lambda_{A_1, q_1}^{\gamma_1, \gamma_2}-\Lambda_{A_2, q_2}^{\gamma_1, \gamma_2}\|^{2\mu_1}
+
\rho^ne^\frac{2\alpha_7}{h}\left|\log\|\Lambda_{A_1, q_1}^{\gamma_1, \gamma_2}-\Lambda_{A_2, q_2}\|\right|^{-2\mu_2} +\frac{1}{\rho^2}.
\]

Choosing $\rho>0$ such that $\rho^n h^2=\frac{1}{\rho^2}$, namely, $\rho=h^{-\frac{2}{n+2}}$, we obtain that
\begin{align*}
\|q\|_{H^{-1}(\mathbb{R}^n)}^2
&\lesssim  
h^\frac{4}{n+2}
+
h^{-\frac{2n}{n+2}}e^\frac{2\alpha_8}{h} \|\Lambda_{A_1, q_1}^{\gamma_1, \gamma_2}-\Lambda_{A_2, q_2}^{\gamma_1, \gamma_2}\|^{2\mu_1}
+ 
h^{-\frac{2n}{n+2}}e^\frac{2\alpha_7}{h}\left|\log\|\Lambda_{A_1, q_1}^{\gamma_1, \gamma_2}-\Lambda_{A_2, q_2}^{\gamma_1, \gamma_2}\|\right|^{-2\mu_2}
\\
&\lesssim  
h^\frac{4}{n+2}
+ 
h^{-\frac{2n}{n+2}}e^\frac{2\alpha_7}{h}\left|\log\|\Lambda_{A_1, q_1}^{\gamma_1, \gamma_2}-\Lambda_{A_2, q_2}^{\gamma_1, \gamma_2}\|\right|^{-2\mu_2}.
\end{align*}
We observe that there exists a constant $\alpha_9>0$ such that $h^{-\frac{2n}{n+2}}e^\frac{2\alpha_7}{h}\le e^\frac{2\alpha_9}{h}$. Hence, the inequality above reads
\[
\|q\|_{H^{-1}(\mathbb{R}^n)}^2
\lesssim 
h^\frac{4}{n+2}
+
e^\frac{2\alpha_9}{h}\left|\log\|\Lambda_{A_1, q_1}^{\gamma_1, \gamma_2}-\Lambda_{A_2, q_2}^{\gamma_1, \gamma_2}\|\right|^{-2\mu_2}.
\]
Thus, by similar arguments as in the end of \cite[Section 4]{Liu_stability}, we get that
\begin{equation}
\label{eq:est_q_in_H-1_part_3}
\|q\|_{H^{-1}(\Omega)} 
\lesssim 
h^\frac{2}{n+2}
+
e^\frac{\alpha_9}{h}\left|\log\|\Lambda_{A_1, q_1}^{\gamma_1, \gamma_2}-\Lambda_{A_2, q_2}^{\gamma_1, \gamma_2}\|\right|^{-\mu_2}.
\end{equation}

To complete the proof of Theorem \ref{thm:estimate_q}, taking $h$ sufficiently small, we argue similarly as in \cite[Section 3]{Liu_2024_biharmonic} again to conclude that there exists a constant $\delta>0$ such that $\|\Lambda_{A_1, q_1}^{\gamma_1, \gamma_2}-\Lambda_{A_2, q_2}^{\gamma_1, \gamma_2}\| <\delta$ implies $h\le h_0$. We then choose
\[
h = \frac{2\alpha_9}{\mu_2} \left|\log\left|\log\|\Lambda_{A_1, q_1}^{\gamma_1, \gamma_2}-\Lambda_{A_2, q_2}^{\gamma_1, \gamma_2}\|\right|\right|^{-1},
\]
and substitute it into \eqref{eq:est_q_in_H-1_part_3} to obtain the inequality
\[
\|q\|_{H^{-1}(\Omega)}
\lesssim 
\left|\log \left| \log\|\Lambda_{A_1, q_1}^{\gamma_1, \gamma_2}-\Lambda_{A_2, q_2}^{\gamma_1, \gamma_2}\|\right|\right|^{\frac{-2}{n+2}}
+
\left|\log\|\Lambda_{A_1, q_1}^{\gamma_1, \gamma_2}-\Lambda_{A_2, q_2}^{\gamma_1, \gamma_2}\|\right|^{-\frac{\mu_2}{2}}.
\]
Due to the inequality $x>\log x$ for $x\gg 1$, we have  
\[
\left|\log\|\Lambda_{A_1, q_1}^{\gamma_1, \gamma_2}-\Lambda_{A_2, q_2}^{\gamma_1, \gamma_2}\|\right|^{-\frac{\mu_2}{2}}
\le
\left|\log \left| \log\|\Lambda_{A_1, q_1}^{\gamma_1, \gamma_2}-\Lambda_{A_2, q_2}^{\gamma_1, \gamma_2}\|\right|\right|^{-\frac{\mu_2}{2}}.
\]
This gives us the estimate
\[
\|q\|_{H^{-1}(\Omega)}
\lesssim 
\left|\log \left|\log\|\Lambda_{A_1, q_1}^{\gamma_1, \gamma_2}-\Lambda_{A_2, q_2}^{\gamma_1, \gamma_2}\|\right|\right|^{-\mu'},
\]
where $\mu' = \min\left\{ \frac{2}{n+2}, \frac{\mu_2}{2}\right\}$.

On the other hand, when $\|\Lambda_{A_1, q_1}^{\gamma_1, \gamma_2}-\Lambda_{A_2, q_2}^{\gamma_1, \gamma_2}\| \geq \delta$, we utilize the continuous inclusions $L^\infty (\Omega) \hookrightarrow L^2(\Omega) \hookrightarrow H^{-1}(\Omega)$ to conclude that
\[
\|q\|_{H^{-1}(\Omega)}
\le 
\|q\|_{L^\infty(\Omega)} 
\le 
M
\lesssim  
\frac{M}{\delta^{\mu'}}\delta^{\mu'}
\lesssim  
\frac{M}{\delta^{\mu'}}\|\Lambda_{A_1, q_1}^{\gamma_1, \gamma_2}-\Lambda_{A_2, q_2}^{\gamma_1, \gamma_2}\|^{\mu'}.
\]
This completes the proof of Theorem \ref{thm:estimate_q}. 

\section{Proof of Theorem \ref{thm:estimate_q_no_A}}
\label{sec:proof_q_no_A}

In this section we assume that the first order term $A=0$ in the perturbed biharmonic operator \eqref{eq:def_biharmonic} and establish a stability estimate for the zeroth order perturbation $q$.
Let us point out that the term $\|A\|_{L^\infty(\Omega)}$ in the estimate \eqref{eq:integral_est_elec_potential_1}, which is of logarithmic type and is established in Theorem \ref{thm:estimate_A}, leads to a log-log-type  stability  for $q$ in Theorem \ref{thm:estimate_q}. In this section we show that the stability improves to one of logarithmic type  for the zeroth order term $q$ when $A=0$. We shall utilize the main ideas from Sections \ref{sec:proof_magnetic} and \ref{sec:proof_q}, and provide the proof for completeness. Similar to Section \ref{sec:proof_q}, let us write $q=q_2-q_1$ again and extend $q$ by zero to $\R^n$. We denote the extension by the same letter. 

Our starting point is the integral estimate \eqref{eq:int_est_q} with $A=0$, in which case we have
\begin{equation}
\label{eq:integral_est_elec_potential_1_A=0}
\left| \int_\Omega qu_2 \overline{u_1}dx\right| \lesssim e^{\frac{4R}{\tau}-\frac{\alpha_1}{3h}} + e^{\frac{6R}{\tau}+\frac{\alpha_2}{3h}} \|\Lambda_{0, q_1}^{\gamma_1, \gamma_2}-\Lambda_{0, q_2}^{\gamma_1, \gamma_2}\|^{\frac{1}{3}}, \quad \tau \to 0,
\end{equation}
where $u_1 \in H^4(\Omega)$ solves the equation $\mathcal{L}_{0,q_1}^*u_1=0$ in $\Omega$, and $u_2 \in H^4(\Omega)$ solves 
\[
\begin{cases}
	\mathcal{L}_{0,q_2}u_2 = 0 \quad \text{in} \quad \Omega, 
	\\
	u_2=f \quad \text{on} \quad \partial \Omega, 
	\\
	\Delta u_2=g \quad \text{on} \quad \partial \Omega.
\end{cases}
\]
By following the arguments from \cite{Choudhury_Krishnan}, the CGO solutions $u_1, u_2\in H^4(\Omega)$ are given by
\[
u_1(x,\zeta_1;\tau)= e^{\frac{ix\cdot \zeta_1}{\tau}}(1+r_1(x,\zeta_1;\tau))
\]
and 
\[
u_2(x,\zeta_2;\tau)= e^{\frac{ix\cdot \zeta_2}{\tau}}(1+r_2(x,\zeta_2;\tau)),
\]
respectively, where the remainder terms $r_1,r_2$ satisfy the estimate \eqref{eq:est_r_domain}. 

We next substitute the CGO solutions above into the left-hand side of \eqref{eq:integral_est_elec_potential_1_A=0} to see that
\[
\int_\Omega qu_2\overline{u_1}dx =  \mathcal{F}(q)(\xi)+I,
\]
where 
\[
I= \int_\Omega qe^{-ix\cdot \xi}(\overline{r_1}+r_2+r_2\overline{r_1})dx.
\]
By the estimate \eqref{eq:est_r_domain} and the Cauchy-Schwarz inequality,  we have
\begin{equation}
\label{eq:est_rem_q_A=0}
|I| \lesssim  \tau, \quad \tau \to 0.
\end{equation}
Thus, in view of estimates \eqref{eq:integral_est_elec_potential_1_A=0} and \eqref{eq:est_rem_q_A=0}, we get that
\begin{equation}
\label{eq:integral_est_elec_potential_3_A=0}
|\mathcal{F}(q)(\xi)| 
\lesssim 
\tau
+
e^{\frac{4R}{\tau}-\frac{\alpha_1}{3h}}
+
e^{\frac{6R}{\tau}+\frac{\alpha_2}{3h}} \|\Lambda_{0, q_1}^{\gamma_1, \gamma_2}-\Lambda_{0, q_2}^{\gamma_1, \gamma_2}\|^{\frac{1}{3}},
\quad  \tau \to 0.
\end{equation}
By setting $\tau = \lambda h$, where $\lambda >0$ is sufficiently large such that both inequalities in \eqref{eq:exponential_tau_h} hold. We deduce from \eqref{eq:integral_est_elec_potential_3_A=0} that
\begin{equation}
\label{eq:integral_est_elec_potential_4_A=0}
\begin{aligned}
|\mathcal{F}(q)(\xi)| 
&\lesssim 
h+
e^{-\frac{\alpha_3}{h}}
+e^\frac{\alpha_4}{h} \|\Lambda_{0, q_1}^{\gamma_1, \gamma_2}-\Lambda_{0, q_2}^{\gamma_1, \gamma_2}\|^{\frac{1}{3}}
\\
& \lesssim h
+e^\frac{\alpha_4}{h} \|\Lambda_{0, q_1}^{\gamma_1, \gamma_2}-\Lambda_{0, q_2}^{\gamma_1, \gamma_2}\|^{\frac{1}{3}}.
\end{aligned}
\end{equation}
Here the constants $\alpha_3,\alpha_4$ are the same as in \eqref{eq:exponential_tau_h}, and we have utilized the inequality $e^{-\frac{\alpha_3}{h}} \le h$ for $h>0$ small.

We now let $\rho >0$ be a parameter, which we shall specify  later. By Parseval's formula, we write
\[
\|q\|_{H^{-1}(\mathbb{R}^n)}^2 \le \int_{|\xi|\le \rho} \dfrac{|\mathcal{F}(q)(\xi)|^2}{1+|\xi|^2}d\xi + \int_{|\xi|\geq \rho} \dfrac{|\mathcal{F}(q)(\xi)|^2}{1+|\xi|^2}d\xi.
\]
We next estimate the two integrals on the right-hand side of the inequality above. For the first term, it follows from the estimate \eqref{eq:integral_est_elec_potential_4_A=0}, as well as the same computations as in the estimate \eqref{eq:est_dA_small_rho}, that
\begin{equation}
\label{eq:est_q_in_H-1_part_1_A=0}
\int_{|\xi|\le \rho} \frac{|\mathcal{F}(q)(\xi)|^2}{1+|\xi|^2}d\xi 
\lesssim 
\rho^n \left(h^2+ e^\frac{2\alpha_4}{h} \|\Lambda_{0, q_1}^{\gamma_1, \gamma_2}-\Lambda_{0, q_2}^{\gamma_1, \gamma_2}\|^{\frac{2}{3}} \right).
\end{equation}
On the other hand, we argue similarly as in the estimate \eqref{eq:est_dA_large_rho} to get the following estimate for the second term:
\begin{equation}
\label{eq:est_q_in_H-1_part_2_A=0}
\int_{|\xi|\geq \rho} \frac{|\mathcal{F}(q)(\xi)|^2}{1+|\xi|^2}d\xi 
\lesssim  
\frac{1}{\rho^2}.
\end{equation}
Hence, by combining \eqref{eq:est_q_in_H-1_part_1_A=0} and \eqref{eq:est_q_in_H-1_part_2_A=0}, we have
\[
\|q\|_{H^{-1}(\mathbb{R}^n)}^2 
\lesssim 
\rho^n h^2+ \rho^n e^\frac{2\alpha_4}{h} \|\Lambda_{0, q_1}^{\gamma_1, \gamma_2}-\Lambda_{0, q_2}^{\gamma_1, \gamma_2}\|^{\frac{2}{3}} +\frac{1}{\rho^2}.
\]

Let us now set $\rho=h^{-\frac{2}{n+2}}$, after which the inequality above becomes
\begin{align*}
\|q\|_{H^{-1}(\mathbb{R}^n)}^2
& \lesssim
h^\frac{4}{n+2} +h^\frac{-2n}{n+2}e^\frac{2\alpha_4}{h}\|\Lambda_{0, q_1}^{\gamma_1, \gamma_2}-\Lambda_{0, q_2}^{\gamma_1, \gamma_2}\|^{\frac{2}{3}}
\\
& \lesssim h^\frac{4}{n+2} +e^\frac{2\alpha_5}{h}\|\Lambda_{0, q_1}^{\gamma_1, \gamma_2}-\Lambda_{0, q_2}^{\gamma_1, \gamma_2}\|^{\frac{2}{3}},
\end{align*}
where the constant $\alpha_5>0$ is the same as in \eqref{eq:est_dA}. 
Thus, it follows from the arguments in \cite[Section 4]{Liu_stability} that
\begin{equation}
\label{eq:est_q_in_H-1_part_3_A=0}
\|q\|_{H^{-1}(\Omega)} 
\lesssim 
h^\frac{2}{n+2} +e^\frac{\alpha_5}{h}\|\Lambda_{0, q_1}^{\gamma_1, \gamma_2}-\Lambda_{0, q_2}^{\gamma_1, \gamma_2}\|^{\frac{1}{3}}.
\end{equation}

We next choose
\[
h= 3\alpha_5\left|\log\|\Lambda_{0, q_1}^{\gamma_1, \gamma_2}-\Lambda_{0, q_2}^{\gamma_1, \gamma_2}\|\right|^{-1},
\]
and substitute it into \eqref{eq:est_q_in_H-1_part_3_A=0} to obtain the estimate
\[
\|q\|_{H^{-1}(\Omega)} 
\lesssim
\|\Lambda_{0, q_1}^{\gamma_1, \gamma_2}-\Lambda_{0, q_2}^{\gamma_1, \gamma_2}\|^\frac{2}{3}
+
\left| \log\|\Lambda_{0, q_1}^{\gamma_1, \gamma_2}-\Lambda_{0, q_2}^{\gamma_1, \gamma_2}\| \right|^\frac{-2}{n+2}.  
\]

When $\|\Lambda_{0, q_1}^{\gamma_1, \gamma_2}-\Lambda_{0, q_2}^{\gamma_1, \gamma_2}\| \geq \delta$, due to the continuous inclusions of spaces $L^\infty (\Omega) \hookrightarrow L^2(\Omega) \hookrightarrow H^{-1}(\Omega)$, we conclude that
\[
\|q\|_{H^{-1}(\Omega)}
\le 
\|q\|_{L^\infty(\Omega)} 
\le 
M
\lesssim  
\frac{M}{\delta^{\frac{2}{3}}}\delta^{\frac{2}{3}}
\lesssim  
\frac{M}{\delta^{\frac{2}{3}}}\|\Lambda_{0, q_1}^{\gamma_1, \gamma_2}-\Lambda_{0, q_2}^{\gamma_1, \gamma_2}\|^{\frac{2}{3}}.
\]
The proof of Theorem \ref{thm:estimate_q_no_A} is now complete.

Finally, Corollaries \ref{cor:estimate_q_Linfty} and  \ref{cor:estimate_q_Linfty_no_A} follow immediately from Theorems \ref{thm:estimate_q} and  \ref{thm:estimate_q_no_A}, respectively, by performing similar computations as in the estimate \eqref{eq:est_dA_Linfty}, where we have utilized the Sobolev embedding theorem and the interpolation inequality.

\section*{Acknowledgments}
We would like to express our deepest gratitude to Katya Krupchyk for her valuable discussions and suggestions. The research of S.S. is partially supported by the National Science Foundation (DMS 2408793). We are also very grateful to the anonymous referees for their valuable feedback, which led to significant improvements of this paper. 

\bibliographystyle{abbrv}
\bibliography{bibliography_biharmonic}

\end{document}